\documentclass[reqno,twoside,12pt]{amsart}

\theoremstyle{plain}

\numberwithin{equation}{section}

\usepackage{comment}
\usepackage{amsmath}
\usepackage{fancybox,color}
\usepackage{latexsym}
\usepackage{mathrsfs}
\usepackage{textcomp}
\usepackage{pxfonts}
\usepackage{esint}
\usepackage{textcomp}
\usepackage{vmargin}
\usepackage{bbm}
\usepackage[active]{srcltx}
\usepackage[colorlinks]{hyperref}

\newtheorem{theo}{Theorem}[section]
\newtheorem{lem}[theo]{Lemma}

\newtheorem{cor}[theo]{Corollary}

\usepackage{graphics}
\input{epsf}

\def\C{{\rm\kern.24em \vrule width.02em height1.4ex
depth-.05ex\kern-.26em C}}

\def\R{{\rm I\hspace{-0.4ex}R}}
\def\B{{\rm I\hspace{-0.4ex}B}}
\def\D{{\rm I\hspace{-0.4ex}D}}

\def\be{\begin{equation}}
\def\ee{\end{equation}}
\def\beq{\begin{equation}}
\def\eeq{\end{equation}}

\usepackage{graphics}
\input{epsf}

\begin{document}

     \title{On $p$-Harmonic Measures in Half Spaces}


\author[J. G. Llorente]{J. G. Llorente}
\address{Departament de Matem\`atiques, Universitat Aut\`onoma de Barcelona, 08193 Bellaterra, Barcelona, Spain}
\email{jgllorente@mat.uab.cat}

\author[J.J. Manfredi]{J. J. Manfredi}
\address{Department of Mathematics, University of Pittsburgh, 301 Thackeray Hall, Pittsburgh, PA 15260}
\email{manfredi@pitt.edu}

\author[W.C. Troy]{W. C. Troy}
\email{troy@math.pitt.edu}

\author[J.M. Wu]{J. M.  Wu}
\address{Department of Mathematics, University of Illinois at Urbana-Champaign, 1409 West Green street, Urbana, IL 61801}
\email{jmwu@illinois.edu}

 \maketitle

\let\thefootnote\relax\footnotetext{\hspace{-7pt}\emph{Keywords:}$p$-laplacian, $p$-harmonic measure, shooting method.

MSC2010: 34B40, 34C11, 35J60.

First author  supported by grants MTM2017-85666-P, 2017 SGR 395 (Spain).

Fourth author  supported by Simon Foundation grant \#353435}

\begin{abstract} For  all $1<p<\infty$ and $N\ge 2$ we prove that there is a constant $\alpha(p,N)>0$ such that the
$p$-harmonic measure in $\R^N_+$ of a ball of  radius $0 < \delta
\leq 1$ in $\R^{N-1}$ is bounded above and below by a constant times
$\delta ^{\alpha (p.N)}$. We provide explicit estimates for the
exponent $\alpha(p,N)$.
\end{abstract}
\noindent




\section{Introduction }

In this paper we study $p$-harmonic measures, which in the case $p=2$, are harmonic functions defined as follows: given a domain $\Omega \subset \R^N$, a point $x\in \Omega$,  and a subset $E\subset
\partial \Omega$, the harmonic measure of $E$ from $x$  in $\Omega$, denoted $\omega (E, x, \Omega )$,  is
the value at $x$ of the harmonic function  $\omega (E, x, \Omega )$ satisfying
\begin{equation*}
\omega (E, x, \Omega )=\left\{
\begin{array}{l}
1  \textrm{ for }x\in E\\
    0  \textrm{ for } x\in\partial\Omega\setminus E
\end{array}
\right.
\end{equation*}
when $\Omega$ and $E$ are sufficiently regular. It follows from the
linearity of the Laplace operator that for fixed $x$, $\omega (. ,
x, \Omega )$ is a probability measure on $\partial \Omega$ and from
the Harnack property $\omega (., x,\Omega ) $ and $\omega (., y ,
\Omega )$ are mutually absolutely continuous.  The study of the
metric properties of harmonic measure and its connection to
Hausdorff measures on $\partial \Omega $ has played a fundamental
role in the development of modern Geometric Function Theory and it
is related to several branches of mathematics as PDE's, Probability,
Potential Theory and Dynamical systems among others.
\par
Our main focus is on the case $p\not=2$ when  the relevant differential equation is the non linear $p$-Laplace equation
 \begin{eqnarray}
\textrm{div}\left( |\nabla u |^{p-2} \nabla u \right) = 0  \, \, \, \, \, \, \, \, (1 < p
< \infty ). \label{plaplace}
\end{eqnarray}
Weak solutions of (\ref{plaplace}) in the Sobolev space
$W^{1,p}_{\textrm{loc}}(\Omega )$ are called \textit{$p$-harmonic} functions
in $\Omega$. When $p \to \infty $ we formally obtain another
differential operator which is not in divergence form, the so called
\textit{infinity laplacian} $\triangle_{\infty}$ given by
\begin{eqnarray}
\triangle_{\infty}u =  \sum_{i,j=1}^N u_{x_i}u_{x_j} u_{x_i x_j}
\label{infty}
\end{eqnarray}
and solutions of (\ref{infty}) in the viscosity sense are
called \textit{infinity harmonic} functions.
\par
For $p\not=2$ the definition of $p$-harmonic measure $\omega_p (.,
x, \Omega ) $ follows the above potential theoretic approach (see
\cite{HKM}). Because of the nonlinearity, $p$-harmonic measures
are  more difficult to handle  and lack some of the nice properties
available in the linear case $p=2$. It is important to recognize that $\omega_p (.)$ is
no longer a measure, not even at the zero level (see \cite{LLMW}).
In the case $p=2$ it is easy to estimate the harmonic measure
of subsets of the boundary of the ball or the upper half-space, due
to the explicit expression of the Poisson kernel in a ball or a
half-space.  The situation is more complicated when $p\neq 2$, even for
simple subsets of the boundary like spherical caps or  half-space balls.
One of the first results on this direction was obtained by Peres,
Schramm, Sheffield, and Wilson (\cite{P}). They proved that
\begin{eqnarray}
\omega_{\infty} (C_{\delta}, 0,\B_N ) \approx \delta^{1/3},
\label{inftymeas}
\end{eqnarray}
where $\B_N$ denotes the unit ball in $\R^N$, $C_{\delta}$ is any
spherical cap of radius $\delta$ and $\omega_{\infty}$ stands for
the $\infty$-harmonic measure. The proof of
(\ref{inftymeas}) is based on two facts. First,  because
of  rotational invariance, the problem can be reduced to  two
dimensions. The second is the use of quasiradial singular $\infty$-harmonic functions
 obtained by Aronsson \cite{Ar1}, which are  of the
  form $r^{k}f(\phi ) $ where $(r,\phi )$ denote polar
coordinates in the plane and $k=-1/3$. The function $r^{-1/3}f(\phi ) $ plays the role of the Poisson kernel in the case
$p=2$ and is used to estimate  $\omega_{\infty} (C_{\delta}, 0,\B_N )$.
\par
Next, we recall properties of the Poisson kernel $P(x,z')$ in the upper
half-space $\R^N_+$ defined by
\begin{eqnarray}
P(x,z') = \frac{x_N}{(|x'-z'|^2 + x_N^2 )^{N/2}} \label{poisson}.
\end{eqnarray}
 Here $z'\in \R^{N-1}$ and  $x=(x', x_N) \in
\R^N_+$, where $x'\in \R^{N-1}$, and $x_N
>0$.
It is well known that $P(., z')$ is positive and harmonic in
$\R^{N}_+$ and has unrestricted boundary values $0$ on $\R^{N-1}
\setminus \{z'\} $ and nontangential limits $+\infty$  when $x\to z'$, for
every $z'\in \R^{N-1}$. Set $z'=0$ and define $P(x) = P(x,0)$.  Since $P$  depends only on the distance to the origin $r =
|x|$ and the azimuth angle $\theta$ formed by $x$ and the positive
$x_N$-axis, we write
$$
P(x) = P(r, \theta ) = r^{-(N-1)} \cos \theta,
$$
so the singularity at the origin is of order $r^{-(N-1)}$.
In the case $p\not=2$ the analogue of the Poission kernel is played
by quasiradial functions  $u = r^k f(\theta )$  with $k<0$ ,  and
the following assumptions on $f$:
\begin{eqnarray}
 & f  :[0, \pi /2] \to \R  \, \, \textnormal{is positive and
 decreasing,}
\label{f1} \\
 & f  \in C^2[0, \pi /2],  \label{f2} \\
 & f(0)  = 1, \,  f'(0) = 0, \, f(\pi /2) = 0 \,  \, \textnormal{and} \, \,  -\infty
< f'(\pi /2) <0. \label{f3}
\end{eqnarray}

\subsection{Previous Results}
Using Aronsson's positive singular quasiradial $p$-harmonic
functions in the plane, Lundström and Vasilis (\cite{L}) obtained
sharp estimates for $p$-harmonic measures in domains satisfying
appropriate regularity assumptions. They proved in particular that

\begin{eqnarray}
\omega_p (I_{\delta}, i, \R^2_+ )  \,  \approx  \, \delta^q  &  \,
\, \,
\, \, \, \text{if} \, \, \, 1 < p < \infty  \vspace{0.2cm}\,\,\,\,\,\,\,\,\,\text{    and }\\
 \omega_p (A_{\delta}, 0, \D ) \,  \approx \,  \delta^q  & \, \, \, \, \,
\,
 \text{if} \, \, \, p\geq 2, 
\end{eqnarray}
where
\begin{eqnarray}
q = \frac{3-p + 2\sqrt{p^2 -3p +3}}{3(p-1)},
\end{eqnarray}
 the interval $I_{\delta} = [-\delta , \delta ] \subset \R$, the arc  $A_{\delta}$
is  of length $\delta$ on the unit circle,  and $\D$ is
 the unit disc.
\par
The case $p=N$ is special because of  conformal invariance.
Hirata (\cite{H}) proved the estimate $\omega_N ( B(\xi, \delta )
\cap \partial \Omega , x_0, \Omega ) \approx \delta$ if $\Omega $ is
$C^{1,1}$. Recently, DeBlassie and Smits (\cite{deb}, \cite{deb2})
have made further contributions in the case $N \geq 3$ and general
$p$. The key point in \cite{deb} is the computation of $\triangle_p
u$ where $u = r^{k}f(\theta )$ is a quasiradial function.
Up to a positive factor, $\triangle_p u $ is given by the
following differential expression
\begin{eqnarray}
[(p-1)(f')^2 +k^2 f^2]f''  + k[(2p-3)k +N-p]f(f')^2 + \nonumber \\
k^3[k(p-1) +N-p]f^3  + (N-2)[(f')^2 +k^2 f^2]f'\cot \theta.
\label{sign1}
\end{eqnarray}
Observe that if  $f$ is  decreasing and $\theta \in [0, \pi /2]$,
the fourth term  in (\ref{sign1}) is negative. Thus, solutions to the reduced
equation
\begin{eqnarray}\label{112}
[(p-1)(f')^2 +k^2 f^2]f''  + k[(2p-3)k +N-p]f(f')^2 + \nonumber \\
k^3[k(p-1) +N-p]f^3  = 0 \label{signredu}
\end{eqnarray}
provide  $p$-superharmonic functions. This approach gives upper
bounds for $p$-harmonic measure in the ball or the upper half-space
\cite{deb} (see section 2 below). In \cite{deb2} another reduction
of \eqref{sign1}
 is used to obtain additional upper and lower bounds under
appropriate restrictions on $N$, $k$ and $p$, which are
 complementary to those of \cite{deb}.
 \par

\subsection{Main Results}

Our first main result provides explicit estimates for the
$p$-harmonic measure of a ball in the boundary of a half-space.
Because of the translation invariance of the $p$-harmonic equation
and the Harnack property we assume that the ball is centered at the
origin and the base point lies at distance one above the center.

\begin{theo}\label{bounds}
Let $N\geq 2$, $1< p < \infty$, $0 < \delta \leq 1$, $B_{\delta}
\subset \R^{N-1}$ the ball in $\R^{N-1}$ centered at the origin of
radius $\delta$ and $x_0 = (0,\cdots, 0,1)$.

\begin{enumerate}
\item If $\boldsymbol{1 < p \leq 3/2}$ then we have
\begin{eqnarray}
\omega_p (B_{\delta}, x_0 , \R^{N}_+ ) \leq C_2 \,
\delta^{\frac{N-1}{p-1}}.  \label{estimate1}
\end{eqnarray}
\item If $\boldsymbol{3/2 <p \leq 2}$ then we have
\begin{eqnarray}
C_1 \, \delta ^{\frac{p+N-3}{2p-3}} \leq \omega_p (B_{\delta}, x_0 ,
\R^{N}_+ ) \leq C_2 \, \delta^{\frac{N-1}{p-1}}. \label{estimate2}
\end{eqnarray}
\item If $\boldsymbol{2\leq p \leq N}$ then we have
\begin{eqnarray}
C_1 \, \delta^{\frac{N-1}{p-1}} \leq \omega_p (B_{\delta}, x_0 ,
\R^{N}_+ ) \leq C_2 \, \delta^{\frac{p+N-3}{2p-3}}. \label{estimate3}
\end{eqnarray}
\item If $\boldsymbol{p\geq N}$ then we have
\begin{eqnarray}
C_1 \, \delta ^{\frac{p+N-3}{2p-3}} \leq \omega_p (B_{\delta}, x_0 ,
\R^{N}_+ ) \leq C_2 \, \delta^{\frac{N-1}{p-1}}. \label{estimate4}
\end{eqnarray}
\end{enumerate}
In each case, $C_1$ and $C_2$ are positive constants only
depending on $N$ and $p$.
\end{theo}

Our second main result is a purely ODE proof of the existence and
uniqueness of a singular quasiradial $p$-harmonic function $r^k
f(\theta )$ in $\R^N_+$, where   $f(\theta)$ satisfies the full
equation
\begin{eqnarray}\label{fullequation}
[(p-1)(f')^2 +k^2 f^2]f''  + k[(2p-3)k +N-p]f(f')^2 + \nonumber \\
k^3[k(p-1) +N-p]f^3  + (N-2)[(f')^2 +k^2 f^2]f'\cot \theta=0.
\end{eqnarray}

\begin{theo}\label{existenceanduniqueness}
Let $1 < p < \infty$ and $N\geq 2$. Then there is a unique $k=
k(p,N)<0 $ and a function $f: [0, \pi /2 ] \to \R$ satisfying
(\ref{f1}), (\ref{f2}) and (\ref{f3}) such that the quasiradial
function $r^k f(\theta )$ is $p$-harmonic in $\R^N_+$.
\end{theo}
\subsection{Remarks:}
\begin{enumerate}
\item Our proof of the existence part in Theorem \ref{existenceanduniqueness} is based on
topological shooting techniques; it is  entirely different from the
previous proofs, relying on PDE methods or complicated techniques
from Harmonic Analysis and PDE's on manifolds ( \cite{LN},
\cite{T83},  \cite{PV09}).

\item This shooting approach
is conceptually simpler  and has the added important advantage
that it is readily applicable to settings in which solutions do not
remain positive and can change sign \cite{ZL03}.\par
\end{enumerate}

Tolksdorff's proof(\cite{T83}) used PDE methods especially suited to
prove the existence of positive $p$-harmonic functions in cones,
which gives also solutions to \eqref{112} when the cone is the
half-space. He also indicated how to get uniqueness from the boundary Harnack inequality. In \cite{PV09}
Tolksdorff's method is extended to
 cover positive  singular solutions. \par
Our proof of the uniqueness  part in Theorem \ref{existenceanduniqueness} is not based on any type of boundary Harnack inequality nor on  Martin boundary estimates. Instead we give a direct proof that is a natural extension of our shooting approach, and is  motivated by the methods of the landmark papers by Kwong~\cite{K} and Coffman~\cite{C}.
\par
By setting $\alpha(p, N)=-k(p,N)$  we obtain the following
consequence of Theorem \ref{existenceanduniqueness} and the results
in section 2:
\begin{cor}
For all $p>1$ and $N\ge 2$  there exist a constant $\alpha(p, N)>0$ such that the $p$-harmonic measure in $\R^N_+$ of a ball of  radius $\delta >0$ in $\R^{N-1}$ satisfies
$$
C_1 \, \delta ^{\alpha(p, N)} \leq
\omega_p (B_{\delta}, x_0 ,\R^{N}_+ ) \leq
C_2 \, \delta^{\alpha(p, N)}. \label{estimate22}
$$
\end{cor}

Finally, we summarize   specific cases where the critical exponent
$k(p,N)$ for equation \eqref{fullequation} with conditions
(\ref{f1}), (\ref{f2}), and (\ref{f3}) is known.

\begin{itemize}
\item $k(2,N) = -(N-1)$ (This corresponds to the case of the Poisson
kernel).
\item $ \displaystyle k(p,2) = - \frac{3-p + 2\sqrt{p^2 -3p
+3}}{3(p-1)}$ (see \cite{Ar1}, \cite{L}).

\item $k(N,N) = -1$, the conformal case, see \cite{H}.

\item $k(\infty, N) = -\frac{1}{3}$, see \cite{P}.

\end{itemize}

Note that in the case $p=1, N=2$ equation \eqref{112} has solutions
$ f(\theta)=\cos^{k}(\theta)$ for all $k$.  When $k<0$ this solution
does not satisfy $f(\pi/2)=0$, so that Theorem
\ref{existenceanduniqueness} does not hold for $p=1$. Apart from
such cases, little is known about the exponent $k(p,N)$.
\par

\par

The organization of the paper is as follows. We show in section 2
how the existence of singular $p$-subharmonic (resp.
$p$-superharmonic) functions in the upper half-space imply lower
(resp. upper) estimates for $p$-harmonic measure (Theorems
\ref{upperbound} and \ref{lowerbound}).  This section contains
standard material based on comparison arguments that  we have
included  for completeness. Section 3 is devoted to the proof of
Theorem \ref{explicitestimates}, which characterizes  the
$p$-subharmonicity or $p$-superharmonicity of a specific quasiradial
test function, depending on the relation between the parameters $N$,
$k$ and $p$. Finally, In section 4 we provide background on ODE shooting methods and give full details of the uniqueness part of the proof  of Theorem
\ref{existenceanduniqueness}.  We have included the full details of the existence proof in the appendix.

\subsection{Remarks}
\begin{enumerate}
\item Theorem \ref{bounds} is a  direct consequence of Theorems \ref{upperbound}, \ref{lowerbound}  and
\ref{explicitestimates}  below.

\item To keep the exposition as simplest as possible we   only
consider  the case of a half-space,  but the same
techniques could be also adapted to cones.
The estimates given by Theorem \ref{bounds}  extend  those of
\cite{deb2} in the case of a half-space.

\item All of our  upper estimates in Theorem \ref{bounds} also
 hold  for spherical caps of the unit ball. However, the
lower estimates in the ball do not directly follow from our method.

\end{enumerate}

\

The research included in this paper originated when the first author
was visiting the departments of Mathematics at the University of
Illinois at Urbana-Champaign and at the University of Pittsburgh. He
wishes to thank both institutions for their support.

\section{Estimating $p$-harmonic measures with non-negative singular quasiradial $p$-super and $p$-subharmonic functions}

The aim of this section is to show that the existence of quasiradial
$p$-superharmonic (resp. $p$-subharmonic) functions $u = r^{k}
f(\theta ) $ in $\R^N_{+}$, where $k<0$ and $f$ satisfies
(\ref{f1}), (\ref{f2}), (\ref{f3}), implies  local upper bounds
(resp. lower bounds) of $p$-harmonic measure on $\partial \Omega$,
provided that $\Omega$ satisfies certain geometrical restrictions.

\subsection{Upper Bound}\label{subsec:upper}

\begin{theo}\label{upperbound}
Suppose that there exist $k<0$ and $f :[0, \pi /2] \to \R $
satisfying (\ref{f1}), (\ref{f2}), (\ref{f3}) such that the
quasiradial function $r^k f(\theta )$ is $p$-superharmonic in
$\R^N_{+}$. Let $\Omega \subset \R^N$ be a convex domain and $x_0
\in \Omega$. Then there exists $C = C(k,f) >0 $ such that for each
$\xi \in
\partial \Omega$ and any $ \delta > 0 $, we have

\begin{eqnarray}
\omega \big( B(\xi , \delta ) \cap \partial \Omega , x_0 , \Omega
\big ) \leq C \, \Big (\frac{\delta}{|x_0 -\xi |} \Big
)^{|k|}.\label{upper}
\end{eqnarray}
\end{theo}

\begin{proof}
From the fact that  $0 \leq \omega \leq 1$ and the invariance of
$p$-harmonic measure by rescaling,  we can assume that $0 < \delta <
|x_0 - \xi| = 1$. Finally, from the convexity assumption on $\Omega$
and the rotational invariance of $p$-harmonic measure, we can also
assume that $\xi = (0,\cdots 0, 2\delta ) $, and that $\Omega
\subset \{(x', x_N )\in \R^N : x_N > 2\delta \} $. We claim that
\begin{eqnarray}
\omega \big( B(\xi , \delta ) \cap \partial \Omega , x_0 , \Omega
\big )  \leq C \, \delta^{|k|} \label{uppersimp}
\end{eqnarray}
for some $C = C(k, f) >0$.

\

We put  $u(x) = C \delta^{|k|}r^{k}f(\theta )$ where $r= |x|$ ,
$\theta$ is the azimuth angle, as introduced at the beginning of the
section, and the constant $C$ will be chosen later. From the
hypothesis, $u$ is $p$-superharmonic in $\R^N_{+}$. We put $\omega
(x)
 \equiv \omega \big( B(\xi , \delta ) \cap \partial \Omega , x , \Omega
\big ) $. We will show that
\begin{eqnarray}
\omega \leq u \, \, \textnormal{on} \, \,  \partial \Omega.
\label{compar}
\end{eqnarray}
Observe that $\omega$ is $p$-harmonic in $\Omega$ and $u$ is
$p$-superharmonic in $\Omega$. Then it  follows from (\ref{compar})
and  the Comparison Principle (\cite{HKM}, Thm. 7.6) that $\omega
\leq u$ in $\Omega$. Therefore
$$
\omega (x_0 ) \leq u(x_0 ) \leq C\delta^{|k|} (1 + 2\delta)^k
\max_{[0, \pi /2]} f \leq  C\delta^{|k|} \, ,
$$
and hence (\ref{uppersimp}); this completes the proof of the
theorem. It remains to show that an appropriate choice of $C$
implies (\ref{compar}). We show next that $\displaystyle C = 3^{|k|}
\big ( f(\pi /6) \big)^{-1}$ works. Since $\Omega \subset \R^N_{+}$,
$\omega \equiv 0 $ on $\partial \Omega \setminus B(\xi , \delta )$
and $u$ is positive, we only need to check (\ref{compar}) on $B(\xi
, \delta ) \cap
\partial \Omega $. Now let $x \in B(\xi , \delta ) \cap \partial
\Omega$. If $r= |x|$ and $\theta$ is the azimuth angle of $x$,
elementary geometry gives that  $r\leq 3\delta$ and $0 \leq \theta
\leq \pi /6 $. By the definition of $u$ and the fact that $f$ is
decreasing we have
$$
u(x) \geq  C\, \delta^{|k|} (3\delta)^{k}f(\pi /6) =  3^k f(\pi /6) \, C,
$$
so the choice $\displaystyle C = 3^{|k|} \big ( f(\pi /6)
\big)^{-1}$ implies that $u(x) \geq 1 = \omega (x) $,  and
(\ref{compar}) follows. This finishes the proof of the theorem.
\end{proof}

\subsection{Lower Bound}\label{subsec:lower}

In this subsection we will restrict our attention to the case where
the domain is the half-space $\R^N_{+}$.

\begin{theo}\label{lowerbound}
Suppose that there exist $k<0$ and $f :[0, \pi /2] \to \R $
satisfying (\ref{f1}), (\ref{f2}), (\ref{f3}) such that the
quasiradial function $r^k f(\theta )$ is $p$-subharmonic in
$\R^N_{+}$. Then there exists a constant $C = C(N,p)>0$ such that
for any ball $B\subset \R^{N-1}$ of radius $0 < \delta \leq 1$ we
have
\begin{eqnarray}
\omega (B, x_B , \R^{N}_{+} ) \geq C \delta^{|k|}, \label{lower}
\end{eqnarray}
where $x_B = (a_B , 1)\in \R^{N}_+$ and $a_B$ is the center of $B$.
\end{theo}

\begin{proof}
We can assume that $B$ is centered at $0$. Denote by $\omega (.)$
the $p$-harmonic measure of $B$ in $\R^{N}_+$. Let $\widetilde{B}$
be the upper half ball in $\R^{N}$ centered at $0$ and of radius
$\delta /2$. Then it follows from the Carleson estimates for
$p$-harmonic measure (\cite{AHSZ}) that there exists $C = C(N,p)>0 $
such that
\begin{eqnarray}
\omega (B, x , \R^{N}_+ ) \geq C \label{carleson}
\end{eqnarray}
whenever $x\in \partial \widetilde{B}$. Now let $\Omega = \R^{N}_+
\setminus \widetilde{B}$. Then $\partial \Omega = F_1 \cup F_2$
where $F_1 = \partial \widetilde{B} \cap \R^{N}_+$ and $F_2 =
\R^{N-1} \setminus \frac{1}{2}B$. (If $t>0$,  $tB$ stands for the
ball concentric to $B$ of radius $t$ times the radius of $B$).

\

We define now $v = C \delta^{|k|} r^k f(\theta)$, which is assumed
to be $p$-subharmonic in $\R^{N}_+$ by the hypothesis. We claim that
$v \leq \omega $ in $\Omega$. Indeed, if $x\in F_1$ then, from
(\ref{carleson}) and the fact that $f\leq 1$, we have
$$
v(x) \leq C \delta^{|k|} \delta^k = \lambda \leq \omega (x).
$$
If $x\in F_2 $ then $v(x) = 0 \leq \omega(x)$. Then by the
Comparison Principle we get $v \leq \omega$ in $\Omega$. Now,
evaluating at $x_0 = (0,\cdots, 0,1)$, we obtain
$$
v(x_0 ) = C \delta^{|k|} 1^k f(0) = \lambda \delta^{|k|} \leq \omega
(B , x_0 , \R^{N}_+ ) \, ,
$$
and (\ref{lower}) follows.
\end{proof}

\section{Explicit choices of  $p$-subharmonic and $p$-superharmonic singular quasiradial functions }
We note that for a  quasiradial function of the form $u = r^k f(
\theta ) $ in $\R^{N}_+$ then the sign of $\triangle_p u$ coincides
with the sign of the differential expression

\begin{eqnarray}
[(p-1)(f')^2 +k^2 f^2]f''  + k[(2p-3)k +N-p]f(f')^2 + \nonumber \\
k^3[k(p-1) +N-p]f^3  + (N-2)[(f')^2 +k^2 f^2]f' \cot \theta .
\label{sign}
\end{eqnarray}

The aim of this section is to seek intervals of the parameter $k<0$
where (\ref{sign}) has a definite sign in the interval $[0, \pi /2]$
and therefore the corresponding quasiradial function $r^k f(\theta
)$ is $p$-super or $p$-subharmonic in $\R^N_+$. We will show that
the choice $f(\theta ) = \cos \theta$ produces  specific intervals
where the sign of (\ref{sign}) in $[0, \pi /2]$ is constant.

\

It is convenient to set $g = \log f$. Then $f' = g'e^g$ and $f'' =
[g'' +(g')^2]e^g$. After cancellation of $e^{3g}$, (\ref{sign}) is
transformed into

\begin{eqnarray}
[(p-1)(g')^2 +k^2][g'' +(g')^2] + k [(2p-3)k +N-p](g')^2 + \nonumber \\
k^3[k(p-1) +N-p] +(N-2)[(g')^2 +k^2]g' \cot \theta.
\label{signsimple}
\end{eqnarray}

Now the choice of $f= \cos \theta$ gives $g = \log ( \cos \theta )$,
$g' = -\tan \theta$ and $g'' = -1-\tan^2 \theta$. Making these
substitutions  in (\ref{signsimple}) we obtain

\begin{eqnarray}
[(p-1)\tan^2 \theta +k^2](-1) +k[(2p-3)k +N-p] \tan^2 \theta + \nonumber \\
k^3[k(p-1) +N-p] + (N-2)(\tan^2 \theta + k^2 )(-1) \, ,
\label{signsimple2}
\end{eqnarray}
and, finally, setting $ t = \tan^2 \theta$ we can write
(\ref{signsimple2}) as
\begin{eqnarray}
\Lambda (k, N, p)(t) \equiv \alpha t + \beta k^2
\end{eqnarray}
where
\begin{eqnarray}
\alpha   = &(2p-3)k^2 +(N-p)k +3-N-p \vspace{0.2cm}\\
\beta  = &  (p-1)k^2 +(N-p)k + 1-N
\end{eqnarray}
and $t\geq 0$. We now study the intervals in which the  sign of
$\Lambda$ is constant in $[0, +\infty)$, in terms of $N$, $k$ and
$p$. It is clear that
\begin{eqnarray*}
\Lambda \geq 0 \, \, \text{in} \, \, [0, +\infty)  & \Leftrightarrow &
\, \alpha \geq 0, \, \beta \geq 0  \, \,\,\,\,\,\text{ and that }   \vspace{0.2cm}\\
\Lambda \leq 0 \, \, \text{in} \, \, [0, +\infty)  & \Leftrightarrow &
\, \alpha \leq 0 , \, \beta \leq 0.
\end{eqnarray*}

As a function of the variable $k$,  $\beta (k)$  has zeros
$\displaystyle - \frac{N-1}{p-1}$ and $1$. As for the function
$\alpha (k)$, unless $p= 3/2$, its zeros are $\displaystyle -
\frac{p+N-3}{2p-3}$ and $1$. Now consider the graphs of the
decreasing functions
$$
f(p) = \frac{p+N-3}{2p-3} \, \, \text { and } \, g(p) = \frac{N-1}{p-1}.
$$
The graph of $f(p)$ has an asymptotic vertical line at $p=3/2$.
Furthermore, the graphs only intersect at $p=2$ and $p=N$. In
particular, $f<0$ in $(0, 3/2)$, $f>g$ in $(3/2 , 2)$, $f<g$ in $(2,
N)$ and $f>g$ in $(N, +\infty)$.

We collect all this information in the following cases.

\begin{enumerate}
\item $1< p \leq 3/2$. We have
 $ \displaystyle \alpha (k) \leq 0$ and  $ \beta (k) \leq 0$  in
$\big [- \frac{N-1}{p-1}, 0 \big ) $

\item $3/2 < p \leq 2$. We have
 $\displaystyle \frac{p+N-3}{2p-3} \geq \frac{N-1}{p-1} $. 
Thus,  we obtain
$\alpha (k), \beta (k) \geq 0$  in $\displaystyle \big( -\infty , -
\frac{p+N-3}{2p-3} \big ]$ and  $ \alpha(k), \beta (k)  \leq 0$ in
$\displaystyle \big[ -\frac{N-1}{p-1}, 0 \big )$.

\item $2\leq p \leq N$.
We have  $ \displaystyle \frac{p+N-3}{2p-3} \leq \frac{N-1}{p-1}$. Thus, we get
$\alpha (k), \beta (k)  \geq 0$ in  $\displaystyle \big ( -\infty ,
- \frac{N-1}{p-1}\big ]$ and $\alpha (k) , \beta (k) \leq 0$ in
$\displaystyle  \big [- \frac{p+N-3}{2p-3}, 0 \big )$.

\item $p\geq N$. Again we have $\displaystyle \frac{p+N-3}{2p-3} \geq \frac{N-1}{p-1}$.
Thus, we get $\alpha (k), \beta (k) \geq 0$ in $\displaystyle \big(-\infty , -
\frac{p+N-3}{2p-3} \big]$ and $ \alpha (k) , \beta (k) \leq 0$ in
$\displaystyle \big [- \frac{N-1}{p-1}, 0 \big )$.

\end{enumerate}

We have therefore proved the following theorem.

\begin{theo}\label{explicitestimates}
Let $k<0$ and consider the quasiradial function $u = r^k \cos
\theta$ in $ \R^{N}_+$.
\begin{enumerate}
\item If $\boldsymbol{1 < p \leq 3/2}$, then $u$ is $p$-superharmonic for $\displaystyle k\in \big [-
\frac{N-1}{p-1}, 0\big )$.
\item If $\boldsymbol{3/2 <p \leq 2}$, then $u$ is $p$-subharmonic for $\displaystyle k \in \big ( -\infty , - \frac{p+N-3}{2p-3}
\big]$ and $p$-superharmonic for $\displaystyle k \in \big
[-\frac{N-1}{p-1}, 0 \big )$.
\item If $\boldsymbol{2\leq p \leq N}$, then $u$ is $p$-subharmonic for $\displaystyle k \in \big ( -\infty , - \frac{N-1}{p-1}
\big]$ and $p$-superharmonic for $\displaystyle k \in \big [ -
\frac{p+N-3}{2p-3}, 0 \big )$.
\item If $\boldsymbol{p\geq N}$, then $u$ is $p$-subharmonic for $\displaystyle k \in \big ( -\infty , - \frac{p+N-3}{2p-3}
\big]$ and $p$-superharmonic for $\displaystyle k \in \big [ -
\frac{N-1}{p-1}, 0\big)$.
\end{enumerate}
\end{theo}

\

\noindent \textbf{Remark.} Assume, for simplicity, that $N\geq 3$
and let $u = r^{k}\cos \theta$ be as above. Even though we are
mainly interested in the case $k<0$, we also record the situation
for $k\geq 0$:
\begin{itemize}
\item If $1 < p < 3/2 $,  then $u$ is $p$-subharmonic for $ \displaystyle k \in [1,
\frac{p+N-3}{3-2p}]$ and $p$-superharmonic for $k\in [0,1]$.
\vspace{0.2cm}
\item If $p\geq 3/2$, then $u$ is $p$-subharmonic for $k\in [1,
+\infty)$ and $p$-superharmonic for $k\in [0,1]$.
\end{itemize}
Observe that the case $k=1$ corresponds to the function $u(x) =
x_N$, which  is $p$-harmonic for any $p>1$.

\

\section{ODE theorems: Proof of Theorem 1.2}
Topological shooting methods were developed for
the  important  model equation
   \begin{eqnarray}
u''+\frac{N-1}{r} u'+  f(u)=0, \label{deq:odddd58}\\
u(0)=u_0>0 \textrm{ and }   u'(0)=0.
 \label{deq:odddd59}
 \end{eqnarray}
Bound  state solutions satisfy
  \begin{eqnarray}
\lim_{r\to \infty}(u(r),u'(r))=(0,0).
 \label{eqdd:odddd}
 \end{eqnarray}
 A positive ground state is a solution  such that $u(r)>0$ for all $ r \ge 0.$
Hastings and McLeod~(\cite{ha},~pp. 5-6) present  a  standard two-step approach to prove that the  initial value problem
  is well posed:
  first, transform    (\ref{deq:odddd58})-(\ref{deq:odddd59})  into the   system
     \begin{eqnarray}
u(r)=u_{0}+\int_{0}^{r}v(\rho)d\rho,~~~~~~~~~~~~~
         \nonumber \vspace{0.1in} \\
 v(r)=-{\frac {1}{r^{N-1}}}
 \int_{0}^{r}\rho^{N-1}f(u(\rho))d \rho.
  \label{eq:oddd66dd3a2}
 \end{eqnarray}
 Second,  an application of    the method of successive approximations
   shows that  for each $u_{0}>0$   there is an interval $[0,r_{1})$ over which
 a unique solution of (\ref{eq:oddd66dd3a2}) exists, and  this solution   depends continuously on $u_{0}.$
   It
  also shows (see~\cite{codd})
      that   $ {\frac {\partial u}{\partial u_{0}}}$ is continuous.
      This  method applies in general (see \cite{codd,hart}), including
 to  problem (\ref{m1addaa})-(\ref{m1addddaa}) studied
 here. Therefore, problem (\ref{m1addaa})-(\ref{m1addddaa}) is well posed.
Since 1963 the investigation of  existence and uniqueness of bound
state  solutions when
  \begin{eqnarray}
 f(u)= |u^{p-1}|u-u, ~N>1~~~{\rm and}~~~1<p<{\frac {N+1}{N-1}}.
 \label{eqdddd:odddd}
 \end{eqnarray}
  has been the central focus of several studies.
 In 1963  Nehari~\cite{neh}   proved
 existence of a positive ground state solution when $p=3$.
 In 1972 Coffman~\cite{C}
  proved  uniquesss of the positive ground state when $p=3$.
In  1989  Kwong~\cite{K} proved uniqueness of the
positive ground state in the general case
 $N>2$ and $1<p<{\frac {N+2}{N-2}}.$
  The uniqueness proofs   in \cite{C,K}
  combine  the fact that
   $u(r,\alpha)$ is continuously differentiable function of $r$ and $\alpha$,
   together with a clever analysis of the equation satisfied    by
    ${\frac {\partial u(r,\alpha)}{\partial \alpha}}.$
In our study of uniqueness we use this same approach to  analyze
properties of ${\frac {\partial y(r,\alpha)}{\partial k}}.$
 Finally, we point out that    in 1990 McLeod, Troy and Weissler~\cite{M3}  let $N>1$ and $1<p<{\frac {N+1}{N-1}},$
  and used a  topological shooting approach to prove existence of sign changing
 bound state solutions. \par
According to (\ref{sign}) and (\ref{f3}) the  initial value problem
to consider is
 \begin{eqnarray}
       \left ( (p-1)(y')^{2}+k^{2}y^{2} \right ) y'' = \left ( (3-2p)k+p-N \right )ky(y')^{2}
 \nonumber \vspace{0.1in} \\
 +\left ( k(1-p)+p-N \right )k^{3}y^{3}  + (2-N) y'\cot(x)\left ( (y')^{2}+k^{2}y^{2} \right ),
 \label{m1addaa}
 \end{eqnarray}
  \begin{eqnarray}
y(0)=1~~{\rm and}~~y'(0)=0.
 \label{m1addddaa}
 \end{eqnarray}
 where we have changed the notation $x$ for $\theta$ and $y(x)$ for $f(\theta)$.
The goal is to prove that if    $N \ge 2$ and $p >1  $   there
exists $k(p,N)<0$ such that if $k=k(p,N)$ then the solution of
(\ref{m1addaa})-(\ref{m1addddaa}) satisfies
  \begin{eqnarray}
0<y(x)<1~{\rm and}~~ y'(x)<0\text{ for all }x \in \left ( 0,{\frac
{\pi}{2}}
  \right ),~~{\rm and} ~ \lim_{ x \to {\frac {\pi}{2}}^{-}}y(x)=0.
 \label{m3wdd4dddd}
\end{eqnarray}
  Because of the known cases (see Introduction) we always assume $N>2$. It suffices to
  consider two separate  regimes:
  {\bf (a)}   $1<p<N,$ and {\bf (b)}  $p\ge N.$
\begin{theo}  \label{th1a}
   {\bf (a)} Let     $1<p<N.$
     There exists a unique $k(p,n) \in \left ( -\infty,{\frac {N-p}{1-p}} \right ) $ such that if $k=k(p,n)$
     then the solution of (\ref{m1addaa})-(\ref{m1addddaa}) satisfies (\ref{m3wdd4dddd}).

 \smallskip \noindent   {\bf (b)} Let    $p\ge N.$
     There exists a unique $k(p,n) \in \left ( -\infty,0 \right ) $ such that if $k=k(p,n)$
     then the solution of  (\ref{m1addaa})-(\ref{m1addddaa}) satisfies (\ref{m3wdd4dddd}).

  \end{theo}





 Our first step in the proof of Theorem~\ref{th1a} is to  follow Deblassie and Smits~\cite{deb}, and set
$
 y={\rm exp} \left( \int_{0}^{x}H(t)dt \right ).
$  Then, we have $y(0)=1$,
and substitution  into
 (\ref{m1addaa}) gives
  the first order equation
 \begin{eqnarray}
       \left ( (p-1)H^{2}+k^{2} \right ) \left (H'+H^{2} \right ) = \left ( (3-2p)k+p-N \right )kH^{2}
 \nonumber \vspace{0.1in} \\
 +\left ( k(1-p)+p-N \right )k^{3}  + (2-N) H \cot(x)\left ( H^{2}+k^{2} \right ),
 \label{m1add}
 \end{eqnarray}
Differentiating and using $y'(0)=0$ we get $H(0)=0$. It now follows from
that property (\ref{m3wdd4dddd})
 holds if $H(x)$ also satisfies
   \begin{eqnarray}
 -\infty< H(x) <0~~{\rm and}~~ H'(x)<0\text{ for all } x \in \left ( 0,{\frac {\pi}{2}}
  \right ),~~{\rm and} ~ \lim_{ x \to {\frac {\pi}{2}}^{-}}H(x)=-\infty.
 \label{m3w4}
 \end{eqnarray}
Thus,    Theorem~\ref{th1a} follows from
 \begin{theo}  \label{th1}
   {\bf (a)} Let   $1<p<N.$
     There exists a unique $k(p,n) \in \left ( -\infty,{\frac {N-p}{1-p}} \right ) $ such that if $k=k(p,n)$
     then the solution of (\ref{m1add}) with $H(0)=0$ satisfies (\ref{m3w4}).

 \smallskip \noindent   {\bf (b)} Let    $p\ge N.$
     There exists a unique $k(p,n) \in \left ( -\infty,0 \right ) $ such that if $k=k(p,n)$
     then the solution of (\ref{m1add}) with $H(0)=0$ satisfies (\ref{m3w4}).

  \end{theo}

 \subsection{Proof of Theorem~\ref{th1} (Existence)}

 We  make use of a topological shooting technique developed in five steps.

 \smallskip \noindent {\bf (Step I)}  First, in
 Lemma~\ref{lem2axx} we prove, for appropriately chosen $k<0,$ that
    solutions of  (\ref{m1add}) with $H(0)=0$ decrease and satisfy $H(x) \to -\infty$ at
  a critical value $x_{k} \in (0,
 \pi].$

 \smallskip \noindent {\bf (Step II)}
We define two topological shooting sets $S_{1}$ and $S_{2}$.\par\noindent
{\bf (a)} For $1<p<N$ we set
  \begin{eqnarray}
  S_{1} = {\Bigg \{ } {\hat k}< {\frac {N-2}{1-p}} \colon ~{\rm if}~  {\hat k}<k< {\frac {N-2}{1-p}}~~{\rm then}~~~~~~~~~~~~~~
       \nonumber \vspace{0.1in} \\
 H'(x)<0~~{\rm and}~~-\infty< H(x) <0\textrm{ for all }  x \in  \left ( 0,{\frac {\pi}{2}} \right ] {\Bigg \} },
 \label{m1ad86}
 \end{eqnarray}
 \noindent {\bf (b)} and for $p\ge N$ we set
   \begin{eqnarray}
  S_{2} = {\Bigg  \{ } {\hat k}< 0 ~{\Bigg |}  ~{\rm if}~~{\hat k}<k< 0~~{\rm then}~~H'(x)<0~~
        \nonumber \vspace{0.1in} \\
   {\rm and}~~-\infty< H(x) <0\textrm{ for all } x \in \left ( 0,{\frac {\pi}{2}} \right ] {\Bigg  \} }.
 \label{m1ad86w}
 \end{eqnarray}
Not that it follows from continuity of solutions with respect to $k$
that $S_{1}$ and $S_{2}$ are open.

  \smallskip \noindent {\bf (Step III)} $S_{1}$ and $S_{2}$ are non empty, open subintervals
  of $(-\infty,0)$ (Lemma~\ref{lem2axxf}).

   \smallskip \noindent {\bf (Step IV)}  $S_{1}$ and $S_{2}$ are bounded below (Lemma~\ref{lem2axxf1})

   \smallskip \noindent {\bf (StepV)}  {\bf (a)}  We have that $k(p,N)=\inf S_{1}$ if $1<p<N$ and
   {\bf (b)} $k(p,N)=\inf S_{2}$  if $p\ge N$ (Lemma~\ref{lem2axxf2} ).
\par\medskip
 \noindent {\bf Remark.} Throughout we   make use of the fact  that
  (\ref{m1add}) can be written as
 \begin{eqnarray}
       \left ( (p-1)H^{2}+k^{2} \right )  H'  =  \left ( H^{2}+k^{2} \right )^{2}(1-p)
  \nonumber \vspace{0.1in} \\
 + \left ( H^{2}+k^{2} \right )
        \left ( k(p-N)+
       (2-N) H \cot(x) \right ).
 \label{m1adad}
 \end{eqnarray}


  \begin{lem}\label{lem2axx}
   {\bf (a)} For $1<p<N$ we have
 \begin{eqnarray}
H'(0)=0 ~~{\rm if}~~k={\frac {N-p}{1-p}}<0,~~~{\rm
and}~~~H'(0)<0~~{\rm if}~~k<{\frac {N-p}{1-p}},
 \label{eq:ode54w34eeew}
 \end{eqnarray}
and for each $k<{\frac {N-p}{1-p}}$ there exists $x_{k} \in (0, \pi ]$ such that
   \begin{eqnarray}
 -\infty< H(x) <0~~~{\rm and}~~ H'(x)<0\textrm{  for all } x \in \left ( 0,x_{k}
  \right ),~~{\rm and} ~ \lim_{ x \to x_{k}^{-}}H(x)=-\infty.
 \label{m3w4e}
 \end{eqnarray}

 \smallskip \noindent   {\bf (b)} For $p\ge N$  we have
  \begin{eqnarray}
 H'(0)<0\textrm{  for all } k <0.
 \label{eq:oded54w34eeew}
 \end{eqnarray}
  For each $k<0$ there exists $x_{k} \in (0, \pi ]$ such that
   (\ref{m3w4e}) holds.

 \end{lem}
 \begin{proof} {\bf (a)} Let $1<p<N.$ It follows from (\ref{m1add}) and $H(0)=0$  that
    \begin{eqnarray}
H'(0)= {\frac {k(1-p)}{N-1}} \left ( k-{\frac {N-p}{1-p}} \right ).
 \label{eq:ode523}
 \end{eqnarray}
Since $1<p<N,$ we conclude from (\ref{eq:ode523}) that
$H'(0)=0$ when $k={\frac {N-p}{1-p}},$ and $H'(0)<0$ when $k<{\frac
{N-p}{1-p}}.$ Next, suppose that there exists   $k<{\frac
{N-p}{1-p}}$ and corresponding
 value $\tilde x \in (0,\pi)$ such that
    \begin{eqnarray}
 ~H'(x)<0\textrm{ for all }  x \in [0, \tilde x), ~~ -\infty < H(\tilde x)<0~~{\rm and}~~ H'(\tilde x)=0.
 \label{eq:odde523}
 \end{eqnarray}
Then
     \begin{eqnarray}
  H''(\tilde x) \ge 0.
 \label{eq:odde523u}
 \end{eqnarray}
  It follows  from a differentiation of (\ref{m1add}) and (\ref{eq:odde523}) that
    \begin{eqnarray}
       \left ( (p-1)H^{2}(\tilde x)+k^{2} \right )H''(\tilde x)   =
       (N-2) \csc^{2} (\tilde x) H(\tilde x)
      \left ( H^{2}(\tilde x)+k^{2} \right )<0,
 \label{m1adaddd}
 \end{eqnarray}
contradicting (\ref{eq:odde523u}). We conclude that $-\infty<H(x)<0$
and $H'(x)<0$ for $x \in (0,\pi)$ as long as the solution exists.
Thus, for each $k<{\frac {N-p}{1-p}}$ we   let $[0,x_{k})$ denote
the maximal subinterval of $[0,\pi]$ over which the solution of
(\ref{m1add}) exists. We need to prove that
\begin{eqnarray}
\lim_{x\to x_{k}^{-}}H(x)=-\infty\textrm { for all } k < {\frac {N-p}{1-p}}.
 \label{eq:ode52ddd36}
 \end{eqnarray}
First, suppose 
 there exists  $k<{\frac {N-p}{1-p}}$ such that $x_{k}\ge\pi.$ We need to prove that
 \begin{eqnarray}
\lim_{x \to \pi^{-}}H(x)=-\infty.
 \label{eq:ode52rddd3dd6}
 \end{eqnarray}
 Suppose, however, that (\ref{eq:ode52rddd3dd6}) does not hold, and that
 \begin{eqnarray}
-\infty<H(x)<0, ~~H'(x)<0 \text{ for all } x \in (0,\pi),\textrm { and } -\infty <  H(\pi)<0.
 \label{eq:ode5236}
 \end{eqnarray}
Dividing  (\ref{m1adad}) by $(p-1)H^{2}+k^{2},$  we conclude, since
$1-p<0$ and $k\,p<0,$  that
\begin{eqnarray}
       H'  \le    {\frac { H^{2}+k^{2}}{(p-1)H^{2}+k^{2}}}
        \left ( -kN+
       (2-N)\, H \cot(x) \right )\text{ for all } x \in (0,\pi).
   \label{m1adaddd9}
 \end{eqnarray}
 We need   upper bounds on the terms on the right side of (\ref{m1adaddd9}).
 For this   define
  \begin{eqnarray}
 L={\rm min} {\Bigg \{} 1, {\frac {1}{p-1}} {\Bigg \} }.
 \label{m1addaaz1}
 \end{eqnarray}
It is easily verified that
  \begin{eqnarray} L\le
 {\frac {H^{2}+k^{2}} { (p-1)H^{2}+k^{2} } } \le {\frac {p}{p-1}}~~~\forall H<0~~{\rm and}~~ k<0.
 \label{m1addaaz1d}
 \end{eqnarray}
 We focus on the interval ${\frac {\pi}{2}} \le x <\pi$ where $-\infty<\cot(x) \le 0.$
First,    define
   \begin{eqnarray}
M_{1}=-{\frac {kNp}{p-1}}>0~~~{\rm and}~~~M_{2}=(2-N)H\left ( {\frac
{\pi}{2}} \right )L>0.
 \label{m101}
 \end{eqnarray}
 Combining (\ref{m1adaddd9}) with (\ref{m1addaaz1}), (\ref{m1addaaz1d}),  and (\ref{m101}),
 we obtain
   \begin{eqnarray}
H' \le M_{1} + M_{2}\cot(x), ~~~{\frac {\pi}{2}} \le x <\pi.
 \label{m101d}
 \end{eqnarray}
Integrating (\ref{m101d}), we conclude that
    \begin{eqnarray}
H(x) \le H \left ( {\frac {\pi}{2}} \right ) +M_{1} \left ( x-{\frac
{\pi}{2}} \right )
 + M_{2}\ln (\sin (x)), ~~~ {\frac {\pi}{2}} \le x < \pi.
 \label{m10d1d}
 \end{eqnarray}
 Since $M_{2}>0,$ it  follows from (\ref{m10d1d}) that $H(x) \to -\infty$ as $x \to \pi^{-},$
 which contradicts (\ref{eq:ode5236}). Thus, it must be the case that (\ref{eq:ode52rddd3dd6}) holds  if $x_{k}=\pi.$
Finally, we need to consider the possibility that
  there exists  $k<{\frac {N-p}{1-p}}$ such that $0<x_{k}<\pi$ and (\ref{eq:ode52ddd36}) does not hold.
  Then
\begin{eqnarray}
-\infty<H(x)<0~~{\rm and}~~H'(x)<0 \textrm{ for all } x \in (0, x_{k}),
 \label{eq:oded5236}
 \end{eqnarray}
 and
 \begin{eqnarray}
-\infty < H(x_{k})<0.
 \label{eq:ode523dd68}
 \end{eqnarray}
 Then  $(p-1)H^{2}(x_{k})+k^{2}>0$ is bounded and  all terms on the right  side of
(\ref{m1adad}) are non zero
 and bounded. Thus,   the solution can be uniquely continued past $x_{k},$ which  contradicts the definition of $x_{k}.$
 This
 completes the proof of property (\ref{eq:ode52ddd36}) and part {\bf (a)}. The proof of part {\bf (b)} is essentially the same and is omitted for the sake of
 brevity.

 \end{proof}
   \begin{lem}\label{lem2axxf}
   {\bf (a)} For $1<p<N$
   there exists $k_{1}<{\frac {N-p}{1-p}}$ such that
 \begin{eqnarray}
k \in S_{1}\text{  for all }  k\in \left ( k_{1},{\frac {N-p}{1-p}} \right
).
 \label{eq:ode579}
 \end{eqnarray}
 \smallskip \noindent   {\bf (b)} For  $p\ge N$
     there exists $k_{2}<0$ such that
\begin{eqnarray}
k \in S_{2}\textrm{  for all }k \in \left ( k_{1},0\right ).
 \label{eq:ode578}
 \end{eqnarray}

 \end{lem}
\begin{proof}
 {\bf (a)}  For $0 \le x \le {\frac {\pi}{2}},$
  as long as $H'(x)<0$ and $-\infty <H(x) \le 0$ we conclude that
  $(2-N)H(x)\cot(x) \ge 0,$ hence (\ref{m1adad}) reduces to
  \begin{eqnarray}
       \left ( (p-1)H^{2}+k^{2} \right )  H'  \ge  \left ( H^{2}+k^{2} \right )\left (    (1-p)
        \left ( H^{2}+k^{2} \right ) +k(p-N) \right ).
 \label{m1adadp}
 \end{eqnarray}
Next,   divide (\ref{m1adadp}) by $ (p-1)H^{2}+k^{2},$ rearrange
terms and obtain
   \begin{eqnarray}
        H'  \ge
    {\frac {   \left ( H^{2}+k^{2} \right )(1-p)}{  (p-1) H^{2}+k^{2}  }}
       \left (
          H^{2}+k \left ( k-   {\frac {N-p}{1-p}} \right ) \right ).
 \label{m1adaddp}
 \end{eqnarray}
 Note that
   \begin{eqnarray}
     k \left ( k-   {\frac {N-p}{1-p}} \right )>0\text{ for all } k < {\frac {N-p}{1-p}}<0.
 \label{m1adad33p}
 \end{eqnarray}
Thus, when $ k < {\frac {N-p}{1-p}}<0$, we  combine
 inequality (\ref{m1addaaz1d}) with (\ref{m1adaddp}) and (\ref{m1adad33p}),
  and  conclude that
    \begin{eqnarray}
        H'  \ge
    -p
       \left (
          H^{2}+k \left ( k-   {\frac {N-p}{1-p}} \right ) \right )
  \label{m1adaddp22}
 \end{eqnarray}
for $x \in  \left [ 0, {\frac {\pi}{2}} \right ]$ as long as
$-\infty<H(x) \le 0$ and $H'(x)<0.$ An integration of
(\ref{m1adaddp22}) gives
    \begin{eqnarray}
        H(x) \ge  -{\sqrt {k \left ( k-   {\frac {N-p}{1-p}} \right ) }}\tan
        \left ( x\,p\, {\sqrt {k \left ( k-   {\frac {N-p}{1-p}} \right ) }} \right )>-\infty
  \label{m1ada5433}
 \end{eqnarray}
for $x \in  \left [ 0, {\frac {\pi}{2}} \right ]$ as long as
$-\infty<H(x) \le 0$ and $H'(x)<0.$ Finally,   let $k_{1} <{\frac
{N-p}{1-p}} $ such that
    \begin{eqnarray}
          p{\sqrt {k \left ( k-   {\frac {N-p}{1-p}} \right ) }} <1~~ \forall k
          \in \left ( k_{1}, {\frac {N-p}{1-p}}  \right ).
  \label{m176}
 \end{eqnarray}
Thus, when $ k \in \left ( k_{1}, {\frac {N-p}{1-p}}  \right ),$ we
conclude  from (\ref{m1ada5433}), (\ref{m176}) and
Lemma~\ref{lem2axx}
 that
   \begin{eqnarray}
             H(x) \ge  -{\sqrt {k \left ( k-   {\frac {N-p}{1-p}} \right ) }}\tan
        \left ( {\sqrt {k \left ( k-   {\frac {N-p}{1-p}} \right ) }}px \right )>-\infty ~~\forall x \in \left [
        0, {\frac {\pi}{2}} \right ],
  \label{m177f}
 \end{eqnarray}
and that
   \begin{eqnarray}
        H'(x)<0 ~~\forall x \in \left [
        0, {\frac {\pi}{2}} \right ].
  \label{m177df}
 \end{eqnarray}
It follows from (\ref{m177f}) and (\ref{m177df}) that $k \in S_{1}$
for all $   k \in \left ( k_{1}, {\frac {N-p}{1-p}}  \right ).$ This
completes the proof of  part {\bf (a)}.

\smallskip \noindent {\bf {Part (b).}}
 There are two small adjustments to make. First, because ${\frac {N-p}{1-p}}\ge 0,$  we replace (\ref{m1adad33p}) with
     \begin{eqnarray}
     k \left ( k-   {\frac {N-p}{1-p}} \right )> 0 \text{ for all } k < 0.
 \label{m1adad33paa}
 \end{eqnarray}
Second, in place of choosing  $k_{1}$ to satisfy (\ref{m176}), we
let $k_{2}<0 $ be chosen such that
  \begin{eqnarray}
          p{\sqrt {k \left ( k-   {\frac {N-p}{1-p}} \right ) }} <1\text{ for all }  k
          \in \left ( k_{2}, 0 \right ).
  \label{m176aa}
 \end{eqnarray}
With these adjustments  the proof of {\bf (b)} is the same as the
proof of {\bf (a)},  and we omit the details for the sake of
brevity.

\end{proof}

   \begin{lem}\label{lem2axxf1}
   {\bf (a)} For $1<p<N$
   there is a value $  k_{3} <{\frac {N-p}{1-p}}<0$ such that $0 <x_{k_{3}} \le {\frac {\pi}{2}}.$

 \smallskip \noindent   {\bf (b)} For $p\ge N$
     there is a value $ k_{4} <0$ such that $0<x_{  k_{4}}\le{\frac {\pi}{2}}.$
 \end{lem}
\begin{proof}
  {\bf (a)}
  We assume, for contradiction that
   \begin{eqnarray}
x_{k} > {\frac {\pi}{2}} \textrm{ for all }  k <{\frac {N-p}{1-p}}.
 \label{eq:ode5765}
 \end{eqnarray}
  Throughout  we  focus   on  the interval $\left [{\frac {\pi}{4}}, {\frac {\pi}{2}}\right ]. $
   First, we observe that
   \begin{eqnarray}
0 \le \cot(x) \le 1 ~~~{\rm and}~~~ 0<{\frac {k^{2}}{
(p-1)H^{2}+k^{2} }} \le 1~~~\forall x \in \left [ {\frac
{\pi}{4}},{\frac {\pi}{2}}\right ].
 \label{eq:ode5789}
 \end{eqnarray}
 When $N>2,$ $p>1,$ $ {\frac {\pi}{4}} \le x \le {\frac {\pi}{2}},$ $H<0$ and $k<0$   we conclude that
   \begin{eqnarray}
  0< {\frac {(2-N)H\cot(x)}{  (p-1)H^{2}+k^{2} }} \le \left ( {\frac {2-N}{2 k{\sqrt {p-1}}    }} \right )
  \left (
  {\frac { 2 k{\sqrt {p-1}}H}{  (p-1)H^{2}+k^{2} }} \right )
  \le {\frac { 2 -N}{   2k{\sqrt {p-1} }}},
 \label{m1ad86dd}
 \end{eqnarray}
since $ \left (  {\sqrt {(p-1)}}   H-k  \right )^{2} \ge 0.$ Next,
divide (\ref{m1adad}) by
 $(p-1)H^{2}+k^{2}$ and get
 \begin{eqnarray}
         H'  =  ( H^{2}+k^{2}) \left (
       \left ( {\frac { (1-p)(H^{2}+k^{2} ) }{(p-1)H^{2}+k^{2}}} \right )
       +   \left ( {\frac {  p-N   }{k}} \right )
  \left ( {\frac {  k^{2}   }{(p-1)H^{2}+k^{2}}} \right ) \right )
  \nonumber \vspace{0.1in} \\
   +  ( H^{2}+k^{2}) \left ( {\frac {  (2-N)H\cot(x)  }{(p-1)H^{2}+k^{2}}} \right )~~~~~~~~~~~~~~~~~~.
 \label{m1adad33}
 \end{eqnarray}
Combining (\ref{m1addaaz1d}), (\ref{eq:ode5765}),  (\ref{m1ad86dd})
and(\ref{m1adad33}, we conclude that
 \begin{eqnarray}
       H'  \le  ( H^{2}+k^{2})
       \left ( (1-p)L+{\frac {  (p-N ) }{k}} + {\frac {  2-N }{2k {\sqrt {p-1}}}} \right ).
  \label{m1adad33d}
 \end{eqnarray}
Now let $\tilde k< {\frac {N-2}{1-p}}<0$ such that
 \begin{eqnarray}
        {\frac {(1-p)L}{2}}+{\frac {  (p-N ) }{k}} + {\frac {  2-N }{2k {\sqrt {p-1}}}}<0\textrm{ for all }  k < \tilde k.
  \label{m1addad33d}
 \end{eqnarray}
It follows from (\ref{m1adad33d}) and (\ref{m1addad33d}) that
\begin{eqnarray}
       H'  \le     {\frac {  (1-p)L }{2}}   ( H^{2}+k^{2})
        \textrm{ for all } x \in  \left [{\frac {\pi}{4}}, {\frac {\pi}{2}}\right ] ~~{\rm and }~~k < \tilde k.
 \label{m1adad3dd3d}
 \end{eqnarray}
Integrating (\ref{m1adad3dd3d}) from ${\frac {\pi}{4}}$ to $x$ gives
\begin{eqnarray}
      \tan^{-1}  \left (   {\frac {  H(x)}{k}} \right )  \ge \tan^{-1}  \left (   {\frac {  H(\pi/4)}{k}} \right )
    +  {\frac {Lk(1-p)}{2}}  \left (x- {\frac {\pi}{4}} \right )
\textrm{ for all }  x \in  \left [{\frac {\pi}{4}}, {\frac {\pi}{2}}\right ].
 \label{m1adad3dd3d2}
 \end{eqnarray}
Assumption (\ref{eq:ode5765}), and the fact that
$H(x)$ is decreasing on $[0, x_{k}],$ imply that
\begin{eqnarray}
    {\frac {\pi}{2}}>  \tan^{-1}  \left (   {\frac {  H(x)}{k}} \right )  \ge
     \tan^{-1}  \left [   {\frac {  H(\pi/4)}{k}} \right ) > 0
\textrm{ for all } x \in  \left [ {\frac {\pi}{4}}, {\frac {\pi}{2}}\right ]~~{\rm and}~~k<\tilde k.
 \label{m1adad3ddd3d2}
 \end{eqnarray}
Thus, when $x={\frac {\pi}{2}},$ it follows from
(\ref{m1adad3dd3d2}) and (\ref{m1adad3ddd3d2}) that
\begin{eqnarray}
    {\frac {\pi}{2}}>  \tan^{-1}  \left (   {\frac {  H(\pi/2)}{k}} \right )  >
    {\frac {L\, k(1-p)}{8} }  \pi  > {\frac {\pi}{2}},
\textrm{ for all }  k< \min {\Bigg \{ } \tilde k, {\frac {4}{L(1-p)}} {\Bigg \} },
 \label{m1a88}
 \end{eqnarray}
a contradiction.  Therefore, we conclude that   (\ref{eq:ode5765})
does not hold, and that there exists $k_{3}< {\frac {N-2}{1-p}}$
such that $x_{k_{3}} \le {\frac {\pi}{2}}.$ The proof of part {\bf
(b)} is essentially the same as the proof of {\bf (a)}, and is
omitted for the sake of brevity.

\end{proof}

   \begin{lem}\label{lem2axxf2}
   {\bf (a)} For $1<p<N$ we have
 \begin{eqnarray}
k(p,N)= \inf S_{1}.
 \label{eq:ode579b}
 \end{eqnarray}

 \smallskip \noindent   {\bf (b)} For  $p\ge N$, we have
 \begin{eqnarray}
 k(p,N)= \inf S_{2}.
 \label{eq:ode578b}
 \end{eqnarray}

 \end{lem}

\begin{proof}
 {\bf (a)}  It  follows from the definition of $S_{1},$
Lemma~\ref{lem2axxf} and Lemma~\ref{lem2axxf1} that $k_{*}=\inf
S_{1}>-\infty.$ Let $H_{*}(x)$ denote the solution of
(\ref{m1add}).  When $k=k_{*}$   Lemma~\ref{lem2axx} implies  that
 $x_{k_{*}} \in \left (0, \pi \right ]$ exists such that
 \begin{eqnarray}
-\infty<H_{*}(x)<0~~{\rm and}~~H_{*}'(x)<0\textrm{ for all } x\in (0,x_{k_{*}}),~~{\rm and}~~\lim_{x \to x_{k_{*}}^{-}}H_{*}(x)=-\infty.
 \label{eq:ode578bd}
 \end{eqnarray}
Our goal is to show that $x_{k_{*}}={\frac {\pi}{2}},$ from which we
conclude that $k(p,N)=k_{*}=\inf S_{1}.$  For contradiction we
assume that   $x_{k_{*}} \ne {\frac {\pi}{2}}.$
 Then $x_{k_{*}}\in \left (0, {\frac {\pi}{2}} \right ) \cup \left ( {\frac {\pi}{2}}, \pi \right ]. $
Suppose that $x_{k_{*}}\in  \left ( {\frac {\pi}{2}}, \pi \right ].
$ Then $k_{*} \in S_{1}$ and  we conclude from the definition of
$S_{1},$ Lemma~\ref{lem2axx} and continuity of solutions with
respect to $k$ over the compact interval $\left [ 0, {\frac
{\pi}{2}} \right ]$ that, if $k_{*}-k>0$ is sufficiently small, then
 \begin{eqnarray}
-\infty<H(x)<0~~{\rm and}~~H'(x)<0\textrm{ for all } x \in \left ( 0, {\frac
{\pi}{2}} \right ].
  \label{eq:ode5ddd78bd}
 \end{eqnarray}
Thus, $k \in S_{1}$ if $k_{*}-k >0$ is sufficiently small, which
contradicts the definition of $k_{*}.$

\smallskip \noindent It remains to consider the second possibility, that $0<x_{k_{*}}<{\frac {\pi}{2}}.$
We define
 \begin{eqnarray}
 \delta={\frac {\pi}{2}}-x_{k_{*}}>0.
  \label{eq:od543ee}
 \end{eqnarray}
We restrict our attention to the interval $\left [ {\frac
{x_{k_{*}}}{2}}, {\frac {\pi}{2}} \right ]$ where
 \begin{eqnarray}
 0 \le \cot(x) \le R=\cot \left ( {\frac {x_{k_{*}}}{2}} \right ) \textrm{ for all } x \in \left [ {\frac {x_{k_{*}}}{2}},
 {\frac {\pi}{2}} \right ].
  \label{eq:oddd5re43ee}
 \end{eqnarray}
It follows from (\ref{eq:ode578bd}) that there is an $\tilde x \in
\left ( {\frac {x_{k_{*}}}{2}}, x_{k_{*}} \right )$ such that
 \begin{eqnarray}
-\infty< H_{*} (\tilde x) <-1,
  \label{eq:oddd31}
 \end{eqnarray}
 \begin{eqnarray}
 0<  {\Bigg | } {\frac {1}{k_{*}}}\left ( {\frac {\pi}{2}}-\tan^{-1} \left ( {\frac { H_{*} (\tilde x)}{k}} \right )  \right ) {\Bigg | }
 <{\frac {\delta}{20}},
  \label{eq:oddd313}
 \end{eqnarray}
and
 \begin{eqnarray}
 0<{\frac {\pi}{2}} \left ( {\frac {|k_{*}||p-N|+k_{*}^{2}|2-p|+(N-2)R}{ (p-1)|H_{*} (\tilde x)|}} \right )<{\frac {\delta}{20}}.
   \label{eq:oddddd31d3}
 \end{eqnarray}
We conclude from
(\ref{eq:oddd31})-(\ref{eq:oddd313})-(\ref{eq:oddddd31d3}), and
continuity of solutions with respct to $k$ over the  interval $\left
[ 0, \tilde x \right ],$ that if $k-k_{*}>0$ is sufficiently small
then $H(x)$ satisfies
  \begin{eqnarray}
-\infty< H (\tilde x) <-1,
  \label{eq:oddd31z}
 \end{eqnarray}
 \begin{eqnarray}
 0< {\Bigg | } {\frac {1}{k}}\left ( {\frac {\pi}{2}}-\tan^{-1} \left ( {\frac { H  (\tilde x)}{k}} \right )  \right ) {\Bigg | }
 <{\frac {\delta}{10}},
  \label{eq:oddd313z}
 \end{eqnarray}
and
 \begin{eqnarray}
 0<{\frac {\pi}{2}} \left ( {\frac {|k_{*}||p-N|+k_{*}^{2}|2-p|+(N-2)R}{ (p-1)|H  (\tilde x)|}} \right )<{\frac {\delta}{10}}.
   \label{eq:oddddd31d3z}
 \end{eqnarray}
Next, recall from Lemma~\ref{lem2axxf} that
 \begin{eqnarray}
 x_{k} > {\frac {\pi}{2}}~~~\forall k \in S_{1}.
   \label{eq:oddd54}
 \end{eqnarray}
It follows from (\ref{eq:oddd54}), and the fact that
$H'(x)<0\textrm{ for all } x \in (0,x_{k_{*}}),$ that
 (\ref{eq:oddd31z})-(\ref{eq:oddd313z})-(\ref{eq:oddddd31d3z}) can be extended to the entire interval
  $\left [ \tilde x, {\frac {\pi}{2}} \right ],$ that is,
    \begin{eqnarray}
-\infty< H ( x) <-1 \textrm{ for all } x \in \left [ \tilde x, {\frac
{\pi}{2}} \right ],
  \label{eq:oddddd31zdd}
 \end{eqnarray}
 \begin{eqnarray}
 0<{\Bigg | } {\frac {1}{k}}\left ( {\frac {\pi}{2}}-\tan^{-1} \left ( {\frac { H  (  x)}{k}} \right )  \right ) {\Bigg | }
 <{\frac {\delta}{10}} \textrm{ for all } x \in \left [ \tilde x, {\frac {\pi}{2}} \right ],
  \label{eq:oddd313zz}
 \end{eqnarray}
and
 \begin{eqnarray}
 0<{\frac {\pi}{2}} \left ( {\frac {|k_{*}||p-N|+k_{*}^{2}|2-p|+(N-2)R}{ (p-1)|H  (  x)|}} \right )<{\frac {\delta}{10}}
\textrm{ for all } x \in \left [ \tilde x, {\frac {\pi}{2}} \right ],
   \label{eq:oddddd31d3zz}
 \end{eqnarray}
when $k-k_{*}>0$ is sufficiently small. In order to make use of
properties
(\ref{eq:oddddd31zdd})-(\ref{eq:oddd313zz})-(\ref{eq:oddddd31d3zz})
 we first write (\ref{m1adad}) as
 \begin{eqnarray}
       \left ( (p-1)H^{2}+k^{2} \right )  H'  =  \left ( H^{2}+k^{2} \right )\left ( (1-p)H^{2} -k^{2} \right )
  \nonumber \vspace{0.1in} \\
 +  (H^{2}+k^{2}) \left (
       k^{2}(2-p)+ k(p-N)+
       (2-N) H \cot(x) \right ).
 \label{m1adadz1}
 \end{eqnarray}
Dividing (\ref{m1adadz1}) by $(p-1)H^{2}+k^{2}$ gives
  \begin{eqnarray}
H'=\left (H^{2}+k^{2} \right ) (-1+F(x)),
  \label{eq:odd32}
 \end{eqnarray}
 where
   \begin{eqnarray}
F(x)= {\frac {k^{2}(2-p)+ k(p-N)+
       (2-N) H(x) \cot(x)}{(p-1)H^{2}(x)+k^{2}}}.
  \label{eq:odddd32}
 \end{eqnarray}
Finally, we divide (\ref{eq:odd32}) by $H^{2}+k^{2}$ and obtain
  \begin{eqnarray}
{\frac {dH}{H^{2}+k^{2}}}=  (-1+F)\, dx.
  \label{eq:odd32fff}
 \end{eqnarray}
Integrating (\ref{eq:odd32fff}) from $\tilde x$ to $x$ gives
 \begin{eqnarray}
 {\frac {1}{k}} \left (\tan^{-1} \left ( {\frac { H  (  x)}{k}} \right )- \tan^{-1} \left ( {\frac { H  ( \tilde x)}{k}} \right )
  \right )  =  \tilde x-x+\int_{\tilde x}^{x}F(t)dt , \textrm{ for all }  x \in \left [ \tilde x, {\frac {\pi}{2}} \right ].
   \label{eq:odew}
 \end{eqnarray}
Setting $x={\frac {\pi}{2}}$ in (\ref{eq:odew}), we obtain
 \begin{eqnarray}
{\frac {\pi}{2}}-{\tilde x}= {\frac {1}{k}} \left (\tan^{-1} \left (
{\frac { H  ( \tilde  x)}{k}}  \right ) -{\frac {\pi}{2}} \right ) +
{\frac {1}{k}} \left (  {\frac {\pi}{2}} -\tan^{-1} \left ( {\frac {
H  ( \pi/2)}{k}} \right ) \right )
   +\int_{\tilde x}^{\pi/2}F(t)   dt
\label{wct}
 \end{eqnarray}
We need upper an bound  on  the right side of (\ref{wct}). First, it
follows from (\ref{eq:oddd313zz}) that
\begin{eqnarray}
  {\Bigg | }{\frac {1}{k}} \left (\tan^{-1} \left ( {\frac { H  ( \tilde  x)}{k}}  \right )
-{\frac {\pi}{2}} \right ) + {\frac {1}{k}} \left (  {\frac
{\pi}{2}} -\tan^{-1} \left ( {\frac { H  ( \pi/2)}{k}} \right )
\right ) {\Bigg | } ~~~~~~~~~~~~~~~~~~~
  \nonumber \vspace{0.1in} \\
   \le  {\Bigg | }{\frac {1}{k}} \left (\tan^{-1} \left ( {\frac { H  ( \tilde  x)}{k}}  \right )
-{\frac {\pi}{2}} \right ) {\Bigg | } +  {\Bigg | }{\frac {1}{k}}
\left (  {\frac {\pi}{2}} -\tan^{-1} \left
 ( {\frac { H  (\pi/2)}{k}} \right ) \right )
{\Bigg | } <  {\frac {\delta}{5}} ~~~~~~ \label{wctab}
 \end{eqnarray}
when $k-k_{*}>0$ is sufficiently small and $\tilde x \le x \le
{\frac {\pi}{2}}.$   Next, we note that $|k| \le |k_{*}|$ and
$|k|^{2}\le |k_{*}|^{2}$ when $k_{*}<  k< {\frac {N-p}{1-p}}<0.$
From these inequalitites, property (\ref{eq:oddd5re43ee}) and the
definition of $F$ given in (\ref{eq:odddd32}) we obtain
    \begin{eqnarray}
| F (x) | \le
 {\frac {|k_{*}|^{2}|2-p|+ |k_{*}||p-N|+
       |2-N| |H(x)|R}{(p-1)H^{2}(x)+k^{2}}}~~~~~~~~~~~~~~~~~~~~~~~~
         \nonumber \vspace{0.1in} \\
 =  {\frac {|k_{*}|^{2}|2-p|+ |k_{*}||p-N|
        }{(p-1)H^{2}(x)+k^{2}}}+{\frac {
       |2-N| |H(x)|R}{(p-1)H^{2}(x)+k^{2}}} \textrm{ for all }  x \in \left [ \tilde x, {\frac {\pi}{2}} \right ].
  \label{eq:odddd3a2}
 \end{eqnarray}
It follows from (\ref{eq:oddddd31zdd}) and the fact that
$H'(x)<0\textrm{ for all }x \in \left [  \tilde x, {\frac {\pi}{2}}
\right],$ that
\begin{eqnarray}
 H^{2}(x) \ge |H(x)| \ge |H(\tilde x)|\textrm{ for all } x \in \left [ \tilde x, {\frac {\pi}{2}} \right].
\label{wctee}
 \end{eqnarray}
We conclude from (\ref{wctee}) that
  \begin{eqnarray}
 {\frac {1} { (p-1)H^{2}(x)+k^{2} } }  \le {\frac {1}{ (p-1)|H(x)| } } \le {\frac {1}{(p-1)|H(\tilde x)| }}
\textrm{ for all }  x \in \left [ \tilde x, {\frac {\pi}{2}} \right].
\label{wcteer}
 \end{eqnarray}
and
  \begin{eqnarray}
 {\frac {|H(x)|} { (p-1)H^{2}(x)+k^{2} } }  \le {\frac {|H(x)|}{ (p-1)|H^{2}(x)| } } \le
   {\frac {1}{(p-1)|H( x)| }} \le {\frac {1}{(p-1)|H(\tilde x)| }}
\label{wcdddteerdd}
 \end{eqnarray}
 for  all $x \in \left [ \tilde x, {\frac {\pi}{2}} \right].$
 Next, we combine (\ref{eq:odddd3a2}),   (\ref{wcteer})and (\ref{wcdddteerdd}), and obtain
    \begin{eqnarray}
| F (x) | \le
 {\frac {|k_{*}|^{2}|2-p|+ |k_{*}||p-N|+
       |2-N| R}{(p-1)|H(\tilde x)| }} \textrm{ for all } x \in \left [ \tilde x, {\frac {\pi}{2}} \right].
 \label{eq:odddd3a2dd}
 \end{eqnarray}
It follows from (\ref{eq:oddd313z}) and (\ref{eq:odddd3a2dd}), and
the fact that ${\Bigg | }  \int_{\tilde x}^{\pi/2} F (t) dt {\Bigg |
} \le   \int_{\tilde x}^{\pi/2} |F (t)| dt$
 that
     \begin{eqnarray}
{\Bigg | }  \int_{\tilde x}^{\pi/2} F (t) dt {\Bigg | }   \le {\frac
{\pi}{2}} \left (
 {\frac {|k_{*}|^{2}|2-p|+ |k_{*}||p-N|+
       |2-N| R}{(p-1)|H(\tilde x)| }} \right )<{\frac {\delta}{10}}.
 \label{eq:odddd3addd2dd}
 \end{eqnarray}
 Finally, we substitute  (\ref{eq:od543ee}),  (\ref{wctab}) and (\ref{eq:odddd3addd2dd}) into equation (\ref{wct}),
  and obtain
      \begin{eqnarray}
\delta={\frac {\pi}{2}}-\tilde x \le {\frac {\delta}{5}}+{\frac
{\delta}{10}}={\frac {3 }{10}}\delta,
 \label{eq:oddd45}
 \end{eqnarray} which is clearly 
 a contradiction
\end{proof}

 \subsection{Proof of Theorem~\ref{th1} (Uniqueness)}
 Our proof, which  relies on  an analysis of the equations satisfied
by
 $U={\frac {H(x)}{k}}$ and     $ {\frac {\partial U}{\partial k}},$  is divided into
    five basic
 steps:

  \medskip \noindent {\bf  (Step I)} First, in     Lemma~\ref{lem2axxf29}  we   prove that
  solutions of (\ref{m1addaa})-(\ref{m1addddaa})
   must
 satisfy
 the   fundamentally important  property
   $   -\infty<y'(\pi/2)<0.$

 \medskip \noindent {\bf  (Step II)} In    Lemma~\ref{lem2axxf278} we analyze the equation satisfied
 by $U(x)={\frac {H(x)}{k}},$ and
  prove  key properies of
  $U(x),$ namely
    \begin{eqnarray}
 U'(x) > 0~~~\forall x \in [0,x_{k})~~~{\rm and}~~~\lim_{x \to x_{k}^{-}}U(x)=\infty,
 \label{eq:ode50}
 \end{eqnarray}
 for each negative $k$ in an approriately chosen  range.
The function $U(x)$ satisfies the equation
  \begin{eqnarray}
       \left ( (p-1)U^{2}+1 \right )  U'  =  \left ( U^{2}+1 \right )^{2}k(1-p)+ \left (U^{2}+1 \right )
        \left (  p-N +
       (2-N) U \cot(x) \right ).
 \label{m1adadu}
 \end{eqnarray}

 \medskip \noindent {\bf (Step III)} In  Lemma~\ref{lem2axxf34d}   we
  determine the behavior of $U$ as $k$ varies.
   Where it is appropriate we write  $U(x,k)$ to emphasize the fact
 that $U$ depends on both $x$ and $k.$    To determine the behavior of $U(x,k)$
 we follow the approach by
  Coffman~\cite{C} and Kwong~\cite{K}, who proved uniqueness of positive ground state solutions of
  \begin{eqnarray}
u''+{\frac {N-1}{r}}u'+u^{p}-u=0,~~u(0)=\alpha>0~~{\rm
and}~~u'(0)=0,
 \label{eq:odddd59}
 \end{eqnarray}
 where $N>1$ and $1<p<{\frac {N+1}{N-1}}.$
  Their proofs of uniqueness combine  the fact that
   $u(r,\alpha)$ is continuously diferentiable function of $r$ and $\alpha,$ together with a clever analysis of
    ${\frac {\partial u(r,\alpha)}{\partial \alpha}}.$
  We  also follow their  approach,
and combine these differentiablity properties together with an
analysis of   the  equation   satisfied by
 ${\frac {\partial U(x,k)}{\partial k}}.$\par
In  Lemma~\ref{lem2axxf34d} and Lemma~\ref{lem2axxf34dd} below  we
  analyze  the behavior of  solutions of the  equation  satisfied by
 ${\frac {\partial U(x,k)}{\partial k}},$  and prove that
   \begin{eqnarray}
 {\frac {\partial U(x,k)}{\partial k}}<0
\textrm{ for all } x \in \left [ 0,x_{k} \right ),
 \label{eq:ode5dd0}
 \end{eqnarray}
  where  $k<0$ is in  an apporopriately chosen range.

  \bigskip \noindent {\bf (Step IV)} We assume, for contradiction that there are two negative $k$ values,
 say $k_{1}<k_{2}<0,$
 and corresponding solutions $y_{1}$ and $y_{2}$ of
  (\ref{m1addaa})-(\ref{m1addddaa}) which satisfy  property  (\ref{m3wdd4dddd}). The key to obtaining
    a contradiction of
   this assumprtion is
  to extend  (\ref{eq:ode5dd0})   to
     \begin{eqnarray}
       {\frac {\partial U(x,k)}{\partial k}}<0\textrm{ for all } x \in \left [ 0, x_{k} \right )~~{\rm and}\textrm{ for all } k \in [ k_{1},k_{2}].
  \label{eq:ode5dddd0}
 \end{eqnarray}
   This is the goal of
  Lemma~\ref{lem2axxf34dd}.

  \bigskip \noindent {\bf (Step V)} We show how to make use of the results  described  in {\bf (I)-(IV)} to
  obtain a contradiction of the assumption that two solutions exist.

\noindent
    \begin{lem}\label{lem2axxf29}
   {\bf (a)} For $1<p<N$
    suppose that $k<{\frac {N-p}{1-p}}<0$ exists
   such that the solution of
    (\ref{m1addaa})-(\ref{m1addddaa}) satisfies property  (\ref{m3wdd4dddd}). Then we have
  \begin{eqnarray}
 -\infty<y' \left ( {\frac {\pi}{2}} \right ) <0.
 \label{eq:ode579by}
 \end{eqnarray}

 \smallskip \noindent   {\bf (b)} For  $p \ge N$
     suppose that $k<0$ exists
   such that the solution of
    (\ref{m1addaa})-(\ref{m1addddaa}) satisfies property  (\ref{m3wdd4dddd}). Then the solution also satisfies
 property (\ref{eq:ode579by}).
\end{lem}
\begin{proof}
{\bf (a)}
      Suppose that $k<{\frac {N-p}{1-p}}<0$ exists
   such that the solution of
    (\ref{m1addaa})-(\ref{m1addddaa}) satisfies  (\ref{m3wdd4dddd}).
Then  it is easily verified that
%
  \begin{eqnarray}
0 \le  {\frac {(p-1)\left (y'\right )^{2}} {
(p-1)(y')^{2}+k^{2}y^{2} } } \le 1
  ~~{\rm and}~~ 0 \le  {\frac {    k^{2}y ^{2}} { (p-1)(y')^{2}+k^{2}y^{2} } } \le 1\,\,
  \text{for all } \in \left [0,x_{k} \right ).
 \label{m1addaazd1z4}
 \end{eqnarray}
Next, we observe that
  \begin{eqnarray}
 (1-p)k+p-N =(1-p) \left (k-{\frac {N-p}{1-p}} \right ) >0.
 \label{m1addaazd}
 \end{eqnarray}
 Combining (\ref{m1addaa}),   (\ref{m1addaazd1z4}),  (\ref{m1addaazd}), and the fact the $0 \le y \le 1$ for
  all $x \in \left [ 0, {\frac {\pi}{2}} \right ],$  gives
  \begin{eqnarray}
        y'' \ge
      -{\frac {2pk^{2}}{p-1}}-k^{2}(p-1)~~~\forall
       x \in \left [ 0, {\frac {\pi}{2}} \right ].
 \label{m1addaa101ee}
 \end{eqnarray}
 An integration of (\ref{m1addaa101ee}) gives
   \begin{eqnarray}
        y' \ge
      - \left ({\frac {2pk^{2}}{p-1}}+k^{2}(p-1) \right )x~~~\forall
       x \in \left [ 0, {\frac {\pi}{2}} \right ].
 \label{m1addaa1d01ee}
 \end{eqnarray}
 From (\ref{m1addaa1d01ee}) it follows that
 \begin{eqnarray}
   y' \left ( {\frac {\pi}{2}} \right ) \ge - \left ({\frac {2pk^{2}}{p-1}}+k^{2}(p-1) \right ){\frac {\pi}{2}}>-\infty.
    \label{m1addddd01ee}
 \end{eqnarray}
 It remains to prove that $  y' \left ( {\frac {\pi}{2}} \right )<0.$  It follows from the first inequality in (\ref{m3wdd4dddd}) that
 $  y' \left ( {\frac {\pi}{2}} \right )  \le 0.$ Suppose that  $  y' \left ( {\frac {\pi}{2}} \right )  = 0.$ Then it follows from
 the fact that $  y  \left ( {\frac {\pi}{2}} \right )  = 0,$ and uniqueness of solutions that $ y(x)=0$ for all
  $  x \in \left [ 0, {\frac {\pi}{2}} \right ], $ contradicting the first inequality in (\ref{m3wdd4dddd}). This proves
(\ref{eq:ode579by})   $1<p<N.$ The proof of
  (\ref{eq:ode579by})   $p \ge N$ is essentially the same and we omit the details.

 \end{proof}

     \begin{lem}\label{lem2axxf278}
   {\bf (a)} For $1<p<N$ and  each
  $k<{\frac {N-p}{1-p}}<0$
  the function $U(x)$ satisfies
  \begin{eqnarray}
 U'(x) > 0\textrm{ for all }x \in [0,x_{k})~~~{\rm and}~~~\lim_{x \to x_{k}^{-}}U(x)=\infty.
 \label{eq:ode579byddd}
 \end{eqnarray}

 \smallskip \noindent   {\bf (b)} For $p \ge N$ and
    ror each  $k<0$
the function $U(x)$  satisfies    (\ref{eq:ode579byddd}).
 \end{lem}
\begin{proof}
 {\bf Proof.}  {\bf (a)}
  First, it follows from $H(0)=0$  that
     \begin{eqnarray}
      U(0)=0\textrm{ for all } k<{\frac {N-p}{1-p}}<0.
 \label{m1adadud}
 \end{eqnarray}
   Lemma~\ref{lem2axx} shows that, for each $k<{\frac {N-p}{1-p}}<0,$ there is an
   $x_{k} \in \left (0, \pi \right ]$  such that
    \begin{eqnarray}
 -\infty< H(x) <0\textrm{ for all }x \in \left ( 0,x_{k}
  \right ),~~{\rm and} ~ \lim_{ x \to x_{k}^{-}}H(x)=-\infty.
 \label{m3w4eddd}
 \end{eqnarray}
We conclude that
     \begin{eqnarray}
 0< U(x) <\infty\textrm{ for all } x \in \left ( 0,x_{k}
  \right ),~~{\rm and} ~ \lim_{ x \to x_{k}^{-}}U(x)=\infty.
 \label{m3w4edddd}
 \end{eqnarray}
 It remains to prove that $U'(x)>0\textrm{ for all }  x \in [0,x_{k}).$
  First, we conclude from (\ref{m1addaazd}), (\ref{m1adadu}) and initial condition (\ref{m1adadud}) that
   \begin{eqnarray}
      U'(0)= {\frac {(1-p)k+p-N}{N-1}}>0\textrm{ for all }k<{\frac {N-p}{1-p}}<0.
 \label{m1adadude}
 \end{eqnarray}
 Suppose, for contradiction, that for some   $k<{\frac {N-p}{1-p}}<0,$
  there exists an $x^{*}_{k} \in (0,x_{k})$ such that
    \begin{eqnarray}
      U'(x) >0\textrm{ for all } x \in \left (0, x^{*}_{k}\right )~~~{\rm and}~~U'\left (x^{*}_{k}\right )=0.
 \label{m1addadude}
 \end{eqnarray}
 We conclude from (\ref{m1addadude}) that
     \begin{eqnarray}
 U''\left (x^{*}_{k}\right ) \le 0.
 \label{m1addaduede}
 \end{eqnarray}
However, it  follows from a differentiation of (\ref{m1adadu}) that
      \begin{eqnarray}
 U''\left (x^{*}_{k}\right ) = \left ({\frac {  U^{2}  \left (x^{*}_{k}\right ) +1}
  {(p-1)U^{2}  \left (x^{*}_{k}\right
 )+1}}\right )
  (N-2) \left (  \csc^{2}   x^{*}_{k} \right )    >0,
 \label{m1adda2duede}
 \end{eqnarray}
 contradicting (\ref{m1addaduede}).
 This completes the proof of part {\bf (a)}.

 \smallskip \noindent {\bf Proof of  (b).}
As before, we have
     \begin{eqnarray}
      U(0)=0\textrm{ for all }  k <0.
 \label{m1adddadud}
 \end{eqnarray}
 Since $p \ge N,$ we
 conclude from (\ref{m1adadude}) that
    \begin{eqnarray}
      U'(0)= {\frac {(1-p)k+p-N}{N-1}}>0\text{ for all }  k<0.
 \label{m1adadudddde}
 \end{eqnarray}
 The remaining details of the proof of {\bf (b)}
 are now the same as   the proof of part {\bf (a),} and are omitted for the sake of
 brevity.

\end{proof}

 \begin{lem}\label{lem2axxf34d}
   {\bf (a)} For $1<p<N$ and  $k<{\frac {N-p}{1-p}}<0$ we have
   \begin{eqnarray}
  {\frac {\partial U(x,k)}{\partial k}}<0\textrm{  for all }x \in \left ( 0, x_{k} \right ).
 \label{eq:oddde579byac}
 \end{eqnarray}

 \smallskip \noindent   {\bf (b)} For $p \ge N$  and  $k<0$
  property (\ref{eq:oddde579byac}) holds.
 \end{lem}

{\bf Proof.}  {\bf (a)}   Following Coffman~\cite{C} and Kwong~\cite{K}, we set
   \begin{eqnarray}
 W={\frac {\partial U(x,k)}{\partial k}},
 \label{eq:oddde579byacd}
 \end{eqnarray}
  take the partial derivative   of (\ref{m1adadu}) and (\ref{m1adadud})
  and
   with respect to $k,$ and obtain
  \begin{eqnarray}
    W'= QW+(1-p)\frac{\left ( U^{2}+1 \right )^{2}}{(p-1)U^{2}+1},~~~W(0)=0,
  \label{m1adadu2}
  \end{eqnarray}
 where
    \begin{eqnarray}
    Q={\frac {(1-p) \left ( 2UU'+4 \left ( U^{3}+U \right ) k \right )+2U(p-N)+(3U^{2}+1)(2-N)\cot(x)}{(p-1)U^{2}+1}}.
 \label{m1ad66}
 \end{eqnarray}
An integration of (\ref{m1adadu2}) gives
     \begin{eqnarray}
W=e^{\int_{0}^{x}Q(t)dt}\int_{0}^{x}(1-p)\frac{(U^{2}(t)+1)^{2}}{(p-1)U^{2}+1}e^{-\int_{0}^{t}Q(u)du}dt<0\textrm{ for all }
x\in (0,x_{k}).
 \label{m1ad66d}
 \end{eqnarray}
 This completes the proof of  {\bf (a).} The proof of {\bf (b)} is  the same as {\bf (a),}
 and  we omit the details for the sake of brevity.

  \begin{lem}\label{lem2axxf34dd}
   {\bf (a)} Let   $N>2$    and $1<p<N,$
   and assume that there exist
   $k_{1}<k_{2}<{\frac {N-p}{1-p}},$  and corresponding solutions $y_{1}$ and $y_{2}$ of
  (\ref{m1addaa})-(\ref{m1addddaa}) which satisfy    (\ref{m3wdd4dddd}).
     Then
         \begin{eqnarray}
  {\frac {\partial U(x,k)}{\partial k}}<0\textrm{  for all }\in [k_{1},k_{2}]\textrm{ and for all } x \in \left ( 0,
    {\frac {\pi}{2}} \right ).
 \label{eq:odddiuyac}
 \end{eqnarray}

 \smallskip \noindent   {\bf (b)} For $p \ge N$
 assume that there exist
   $k_{1}<k_{2}<0,$  and corresponding solutions $y_{1}$ and $y_{2}$ of
  (\ref{m1addaa})-(\ref{m1addddaa}) which satisfy     (\ref{m3wdd4dddd}).
  Then property
   (\ref{eq:odddiuyac}) holds.

 \end{lem}

\begin{proof}  {\bf (a)}
 The first step in proving (\ref{eq:odddiuyac}) is to define
     \begin{eqnarray}
     S_{3}=  {\Bigg \{} \hat x \in \left ( 0,  {\frac {\pi}{2}} \right )~{\Big |}~U(x,k)~~{\rm exists}~~
\textrm{ for all } k \in [k_{1},k_{2}]~~{\rm and}\textrm{ for all } x \in \left ( 0, \hat x \right ] {\Bigg \}}.
 \label{eq:oddddiuyac}
 \end{eqnarray}
 We need to prove that
      \begin{eqnarray}
   {\rm S_{3}~is~open,}~~S_{3} \ne \phi~~{\rm and}~~S_{3}=\left ( 0,  {\frac {\pi}{2}} \right ).
 \label{eq:oddddddiuyac}
 \end{eqnarray}
 First, we show that $S_{3}$ is open. Let $\bar x \in S_{3}.$ We need to prove that there exists $\bar \epsilon>0$ such that
       \begin{eqnarray}
 [\bar x,\bar x+\bar \epsilon) \subset S_{3}.
 \label{eq:o33ddac}
 \end{eqnarray}
 Since $[k_{1},k_{2}]$ is compact, it    follows  from  Lemma~\ref{lem2axxf278} and continuity of solutions with respect to
 intial conditions and parameters that there exist $\delta>0$ and $\bar \epsilon >0$ such that
        \begin{eqnarray}
 \delta <U'(x,k) <\infty~~~\forall x \in [\bar x,\bar x+ \bar \epsilon]~~{\rm and}~~\forall k \in [k_{1},k_{2}].
 \label{eq:odd33ddac}
 \end{eqnarray}
 Thus, $U(x,k)$ exists   fora all $x \in [\bar x,\bar x+ \bar \epsilon)$ and
  for all $ k \in [k_{1},k_{2}],$ hence
 $\left [\bar x,\bar x+\bar \epsilon \right ) \subset S_{3}$ as claimed.

\smallskip \noindent  Next, we show that $S_{3} \ne \phi.$
  Ignoring the negative terms in (\ref{m1adadu}), we conclude that
 \begin{eqnarray}
       \left ( (p-1)U^{2}+1 \right )  U' \le -kp\left ( U^{2}+1 \right )^{2} + p\left (U^{2}+1 \right ).
 \label{m1adaduzz}
 \end{eqnarray}
 Dividing both sides by  by $(p-1)U^{2}+1,$ and using the fact that $-k\,p \le -k_{1}\,p,$ gives
  \begin{eqnarray}
      U' \le {\frac {U^{2}+1}{(p-1)U^{2}+1}}
    \left (  -pk_{1} \left ( U^{2}+1 \right )  + p \right ).
 \label{m1adaduzddz}
 \end{eqnarray}
Next, note that we have
    \begin{eqnarray}
 0< {\frac {U^{2}+1} { (p-1)U^{2}+1 } } \le {\frac {p}{p-1}}.
  \label{m1ddad111d}
 \end{eqnarray}
 Combining (\ref{m1adaduzddz}) and (\ref{m1ddad111d}), we conclude that
   \begin{eqnarray}
      U' \le
      {\frac {p^{2}}{p-1}}
    \left (  -k_{1} \left ( U^{2}+1 \right )  + 1\right )
     \le
      {\frac {p^{2}}{p-1}}
    \left (  -2k_{1}    + 1\right )
   \label{m1ad55dd5}
 \end{eqnarray}
 for $x \in \left [ 0, {\frac {\pi}{2}} \right )$ as long as $0\le U(x)\le 1.$
  An integration gives
    \begin{eqnarray}
      U(x)
        \le 1~~\forall x \in
        \left [0, {\frac {p-1}{p^{2}\left ( -2k_{1} +1\right )}}\right ]~~{\rm and}~~\forall k \in [k_{1},k_{2}].
 \label{m1a4d557dd5}
 \end{eqnarray}
Because $p>1,$ it   is easily verified that
  $ {\frac {p-1}{p^{2}\left [ -2k_{1} +1\right )}}<{\frac {\pi}{2}}.$
This fact and  (\ref{m1a4d557dd5}) imply that
     \begin{eqnarray}
  \left (0, {\frac {p-1}{p^{2}\left [ -2k_{1} +1\right )}}\right ]  \subset S_{3}.
 \label{m1a4dd557dd5}
 \end{eqnarray} This proves that $S_{3} \ne \phi.$ It remains to
 show  that $S_{3}=\left (0, {\frac {\pi}{2}} \right ).$ Suppose, however that
      \begin{eqnarray}
  x_{*}=\sup S_{3}<{\frac {\pi}{2}}.
 \label{m1a4ddd557dd5}
 \end{eqnarray}
 Then there exists $k_{*} \in \left (k_{1},k_{2} \right )$ such that $U(x,k_{*})$ ceases to exist at $x_{*},$ hence
   it must be the
 case that $x_{k_{*}}=x_{*}$ and $U(x,k_{*}) \to \infty$ as $x \to x_{*}^{-}.$ Since $U(x,k_{1})$ is finite   for all $x \in [0,x_{*}],$ we conclude that
       \begin{eqnarray}
  \lim_{x \to x_{*}^{-}}\left (U(x,k_{*})-U(x,k_{1})\right )=\infty.
 \label{m1a4dddd557dd5}
 \end{eqnarray}
 At $x=0,$ it follows from    (\ref{m1adadu}) and (\ref{m1adadud}) that
      \begin{eqnarray}
      U(0,k_{*})-U(0,k_{1})=0~~~{\rm and}~~~U'(0,k_{*})-U'(0,k_{1})=(k_{*}-k_{1})\frac{(1-p)}{N-1}<0.
 \label{m1ad90}
  \end{eqnarray}
 From (\ref{m1a4dddd557dd5}) and (\ref{m1ad90}) we conclude that there is an $\tilde x \in (0,x_{*})$ such that
       \begin{eqnarray}
      U(x,k_{*})-U(x,k_{1})<0~~~\forall x \in (0,\tilde x)~~~{\rm and}~~~U(\tilde x,k_{*})-U(\tilde x,k_{1})=0.
 \label{m1add902}
  \end{eqnarray}
 However, since $\tilde x<x_{*},$ it follows from  from (\ref{eq:odddiuyac}) and continuity  that
        \begin{eqnarray}
 U(\tilde x,k_{*})-U(\tilde x,k_{1})=\int_{k_{1}}^{k_{*}} {\frac {\partial U(\tilde x,k)}{\partial k}}dk<0,
 \label{m1addd90d2}
  \end{eqnarray}
 contradicting (\ref{m1add902}). We conclude that
  \begin{eqnarray}
 S_{3}=\left (0, {\frac {\pi}{2}} \right )
 \label{eq:oddddddtac}
 \end{eqnarray}
 as claimed.
   Finally, it follows from inequality (\ref{eq:oddde579byac}), the definition of $S_{3},$  (\ref{eq:oddddddtac}) and continuity that
  \begin{eqnarray}
  {\frac {\partial U(x,k)}{\partial k}}<0~~~\forall k \in [k_{1},k_{2}]~~{\rm and}~~\forall x \in \left ( 0,
    {\frac {\pi}{2}} \right ).
 \label{eq:odddiduyac}
 \end{eqnarray}

\end{proof}

\bigskip \noindent {\bf  The Final Step (V)} \par
 \smallskip \noindent {\bf Case (a)} Let $1<p<N$
   and assume that there exist
   $k_{1}<k_{2}<{\frac {N-p}{1-p}}<0,$  and corresponding solutions $y_{1}$ and $y_{2}$ of
  (\ref{m1addaa})-(\ref{m1addddaa}) which satisfy     (\ref{m3wdd4dddd}). It follows from (\ref{m1addddaa}), (\ref{m3wdd4dddd})
   and Lemma~\ref{lem2axxf29} that there exist values $\lambda_{1}>0$ and $\lambda_{2}>0$ such that
        \begin{eqnarray}
 y_{1}(0)=y_{2}(0)=1,~y_{1}'(0)=y_{2}'(0)=0,
   \label{eq:oerrdddiduyac}
 \end{eqnarray}
     \begin{eqnarray}
 y_{1} \left ( {\frac {\pi}{2}}\right )= y_{2} \left ( {\frac {\pi}{2}}\right )=0,~
  y_{1}' \left ( {\frac {\pi}{2}}\right )=-\lambda_{1}<0~~{\rm and}~~y_{2}' \left ( {\frac {\pi}{2}}\right
  )=-\lambda_{2}<0.
 \label{eq:oedddiduyac}
 \end{eqnarray}
 Next, we conclude from (\ref{eq:odddiuyac}) and continuity that
  \begin{eqnarray}
 U(x,k_{2})-U(x,k_{1})=\int_{k_{1}}^{k_{2}}{\frac {\partial U(x,k)}{\partial k}}dk<0~~~\forall
  x \in \left ( 0, {\frac {\pi}{2}} \right ).
  \label{eq:oddde579byab}
 \end{eqnarray}
  Substituting  $U(x,k_{1})={\frac {1}{k_{1}}}{\frac {y_{1}'(x)}{y_{1}(x)}}$ and
  $U(x,k_{2})={\frac {1}{k_{2}}}{\frac {y_{2}'(x)}{y_{2}(x)}}$  into (\ref{eq:oddde579byab}), we obtain
    \begin{eqnarray}
    {\frac {1}{k_{1}}}{\frac {y_{1}'(x)}{y_{1}(x)}}>{\frac {1}{k_{2}}}{\frac {y_{2}'(x)}{y_{2}(x)}}>0~~~\forall
  x \in \left ( 0, {\frac {\pi}{2}} \right ).
  \label{eq:odddedd}
 \end{eqnarray}
   An integration, together with (\ref{eq:oerrdddiduyac}),  gives
    \begin{eqnarray}
    y_{1}(x)-\left (y_{2}(x)\right )^{{\frac {k_{1}}{k_{2}}}}\le 0~~~\forall
  x \in \left [0, {\frac {\pi}{2}} \right )
  \label{eq:dddodddedd}
   \end{eqnarray}
  Since ${\frac {k_{1}}{k_{2}}}>1,$ it follows from (\ref{eq:oedddiduyac})   that
      \begin{eqnarray}
    {\frac {d}{dx}} \left (y_{1}(x)-\left (y_{2}(x)\right )^{{\frac {k_{1}}{k_{2}}}} \right ) {\Bigg |}_{x={\frac
    {\pi}{2}}}=-\lambda_{1}<0.
  \label{eq:ddddodddedd}
 \end{eqnarray}
We conclude from  (\ref{eq:oedddiduyac}) and (\ref{eq:ddddodddedd})
there is an $\epsilon>0$ such that
     \begin{eqnarray}
    y_{1}(x)-\left (y_{2}(x)\right )^{{\frac {k_{1}}{k_{2}}}}>0\textrm{  for all }  x \in \left (  {\frac {\pi}{2}}-\epsilon, {\frac {\pi}{2}} \right ),
  \label{eq:dd3dodddedd}
   \end{eqnarray}
  contradicting (\ref{eq:dddodddedd}). Thus,   $k_{1}$ and $k_{2}$ cannot exist and the proof is complete.

 \smallskip \noindent {\bf Case (b)} For $ p \ge N,$
    assume that there exist
   $k_{1}<k_{2}<0,$  and corresponding solutions $y_{1}$ and $y_{2}$ of
  (\ref{m1addaa})-(\ref{m1addddaa}) which satisfy     (\ref{m3wdd4dddd}). A contradiction is obtained in exactly the same
  way as in {\bf Case (a)}, and we omit the details for brevity.

 \end{document}